\documentclass{amsart}

\usepackage[T1]{fontenc}
\usepackage[utf8]{inputenc}
\usepackage[english]{babel}
\usepackage{graphicx}
\usepackage{amssymb}
\usepackage{mathtools}
\usepackage{tikz-cd} 
\usepackage{mathrsfs}
\usepackage{pb-diagram}
\usepackage{yfonts}
\usepackage{color}
\usepackage{verbatim}
\usepackage{dsfont}
\usepackage{enumitem}
\usepackage[normalem]{ulem}
\usepackage{etoolbox}
\usepackage{hyperref}

\definecolor{mred}{rgb}{0.83, 0.0, 0.0}
\definecolor{darkspringgreen}{rgb}{0.09, 0.45, 0.27}
\definecolor{ruby}{rgb}{0.88, 0.07, 0.37}

\theoremstyle{plain}

\theoremstyle{definition}

\def\R{\mathbb{R}}
\def\Q{\mathbb{Q}}
\def\C{\mathbb{C}}
\def\Z{\mathbb{Z}}
\def\N{\mathbb{N}}
\def\O{\mathcal{O}}

\def\nm{\lVert\cdot\rVert}

\def\deg{{\mathrm{deg}}}

\def\ot{\otimes}

\def\shfL{\mathcal{L}}
\def\shfM{\mathscr{M}}

\def\ardeg{\widehat{\mathrm{deg}}}

\def\ovl{\overline}

\def\scrX{\mathcal{X}}

\def\colorsout#1{\bgroup\markoverwith{\textcolor{#1}{\rule[0.5ex]{2pt}{0.7pt}}}\ULon}
\def\coloruline#1{\bgroup\markoverwith{\textcolor{#1}{\rule[-0.5ex]{2pt}{0.7pt}}}\ULon}

\def\mO{\mathcal{O}}

\newcommand{\mb}[1]{\mathbb{#1}}
\newcommand{\mr}[1]{\mathrm{#1}}
\newcommand{\mc}[1]{\mathcal{#1}}
\newcommand{\ov}{\overline}
\newcommand\ra{\rightarrow}

\newcommand{\cra}{\stackrel{\sim}\rightarrow}
\newcommand\wt{\widetilde}
\newcommand\Spe{\mathrm{Spec\,}}
\newcommand\dra{\dashrightarrow}
\newcommand\reldeg{\mathrm{reldeg}}
\newcommand\di{\mathrm{div}}

\theoremstyle{thmstyleone}
\newtheorem{theo}{Theorem}[section]
\newtheorem{coro}[theo]{Corollary}

\newtheorem{prop}[theo]{Proposition}
\newtheorem{lemm}[theo]{Lemma}
\newtheorem{rema}[theo]{Remark}
\newtheorem{conj}[theo]{Conjecture}
\newtheorem{exem}[theo]{Example}

\theoremstyle{definition}
\newtheorem{defi}[theo]{Definition}
\newtheorem{defiprop}[theo]{Definition-Proposition}

\makeatletter
\newcommand\tint{\mathop{\mathpalette\tb@int{t}}\!\int}
\newcommand\bint{\mathop{\mathpalette\tb@int{b}}\!\int}
\newcommand\tb@int[2]{%
  \sbox\z@{$\m@th#1\int$}%
  \if#2t%
    \rlap{\hbox to\wd\z@{%
      \hfil
      \vrule width .35em height \dimexpr\ht\z@+1.4pt\relax depth -\dimexpr\ht\z@+1pt\relax
      \kern.05em 
    }}
  \else
    \rlap{\hbox to\wd\z@{%
      \vrule width .35em height -\dimexpr\dp\z@+1pt\relax depth \dimexpr\dp\z@+1.4pt\relax
      \hfil
    }}
  \fi
}
\makeatother

\newcommand*\suppresschapternumber{%
  \let\@makechapterhead\@makeschapterhead
  \patchcmd{\@chapter}
    {\protect\numberline{\thechapter}}
    {}
    {}{}%
}
\newcommand*\removedotbetweenchapterandsection{%
  \renewcommand\thesection{\thechapter\@arabic\c@section}%
}

\numberwithin{equation}{section}

\title{Arithmetic degrees of dynamical systems over fields of characteristic zero}
\author{Wenbin LUO}
\author{Jiarui SONG}
\address{School of Mathematical Sciences, Shanghai Key Laboratory of PMMP, East China Normal University, 500 Dongchuan Road, Shanghai 200241, People’s Republic of China}
\email{luowenbin\_math@outlook.com}
\address{School of Mathematical Sciences, Peking University, Road Yiheyuan No.5, Beijing 100871, China}
\email{soyo999@pku.edu.cn}

\date{\today}

\begin{document}
\begin{abstract}
In this article, we generalize the arithmetic degree and its related theory to dynamical systems defined over an arbitrary field $\mathbf{k}$ of characteristic $0$. We first consider a dynamical system $(X,f)$ over a finitely generated field $K$ over $\mathbb{Q}$, where we introduce the arithmetic degrees $\alpha(f,\cdot)$ for $\overline{K}$-points by using Moriwaki heights. We study the arithmetic dynamical degree of $(X,f)$ and establish the relative degree formula. The relative degree formula gives a proof of the fundamental inequality, that is, the upper arithmetic degree $\overline{\alpha}(f,x)$ is less than or equal to the first dynamical degree $\lambda_1(f)$ in this setting. By taking spread-outs, we extend the definition of arithmetic degrees to dynamical systems over the field $\mathbf k$. We demonstrate that our definition is independent of the choice of the spread-out. Moreover, in this setting, we prove certain special cases of the Kawaguchi--Silverman conjecture. A main novelty of this paper is that, we give a characterization of arithmetic degrees of "transcendental points" in the case $\mathbf{k}$ is a universal domain of characteristic zero, from which we deduce that $\alpha(f,x)=\lambda_1(f)$ for very general $x\in X(\mathbf{k})$ when $f$ is an endomorphism.
\end{abstract}

\maketitle

\section*{Introduction}\label{sec_intro}
Algebraic dynamical systems focus on self-maps of an algebraic variety and their iterations. For example, let $f: X\dashrightarrow X$ be a dominant rational self-map of a projective normal variety $X$ over a field $\mathbf{k}$. We have some numerical estimates on the complexity of $f$. For each $n\in \N_+$, consider the following resolution of indeterminacy:
\[
\begin{tikzcd}
 & \Gamma_{f^n} \arrow{dl}[swap]{\pi_{n,1}} \arrow{dr}{\pi_{n,2}}\\
X \arrow[dashed]{rr}{f^n} && X ,
\end{tikzcd}
\]
where $\Gamma_{f^n}$ is the normalization of the graph of $f^n$ in $X\times X$, $\pi_{n,1}$ is a birational morphism to the first coordinate, and $\pi_{n,2}$ is a generically finite morphism to the second coordinate. Fix a big and nef line bundle $H$ on $X$. The $k$-th \textit{dynamical degree} is defined as \begin{equation}\label{eq_dyn_deg}\lambda_k(f):=\lim\limits_{n\rightarrow\infty}(\pi_{n,2}^* H^k\cdot \pi_{n,1}^* H^{\dim X-k})^{1/n},
\end{equation}
where the convergence of the limit is proved in \cite[Corollary 7]{DS2005Une}, \cite[Theorem 1.1]{Tru2020Rela}, and \cite[Theorem 1]{dang_degrees_2020}.

On the other hand, for any $x\in X(\ovl {\mathbf k})$, we consider the behavior of $f^n(x)$ as $n\rightarrow+\infty$. We denote 
$$X_f(\ovl{\mathbf k}):=\{x\in X(\ovl{\mathbf k})\mid f\text{ is well-defined at } f^n(x)\text{ for each }n\in\N\}.$$ We are mainly interested in the case that the $f$-orbit $\mO_f(x):=\{f^n(x)\mid n\in\N\}$ is Zariski dense in $X$.
If $\mathbf{k}$ moreover admits some arithmetic structures, for example, $\mathbf{k}=\ovl {\mb{Q}}$ or $\mathbf{k}$ is a function field $\mathsf{k}(t)$ where $\mathsf k$ is an arbitrary field, we may consider the \textit{arithmetic degrees}
\begin{equation}\label{eq_arith_deg}
    \begin{aligned}
    \ovl\alpha(f,x):=\limsup\limits_{n\rightarrow \infty}(\max\{h(f^n(x)),1\})^{1/n},\\
    \underline\alpha(f,x):=\liminf\limits_{n\rightarrow \infty}(\max\{h(f^n(x)),1\})^{1/n},
\end{aligned}
\end{equation} where $h(\cdot)$ is a Weil height function associated with an ample divisor. 
If $\ov\alpha(f,x)=\underline \alpha(f,x)$, then we simply denote it by $\alpha(f,x)$.
The Kawaguchi--Silverman conjecture proposed in \cite{KS2016dyna} asserts that $\ov\alpha(f,x)=\underline \alpha(f,x)$, and $\alpha(f,x)=\lambda_1(f)$ if the $f$-orbit $\mO_f(x)$ is Zariski-dense in $X$.

For $\mathbf{k}=\ovl \Q$, the dynamical system must be defined over some number field, excluding the study of transcendental points. If $\mathbf{k}$ is a function field, the height function is typically the geometric height \cite[Definition 2.2]{xie2023poschar}, lacking the Northcott property if the base field is infinite, leading to a counterexample \cite[Example 3.8]{Matsuzawa2018zero} of the Kawaguchi--Silverman conjecture. 

In \cite{moriwaki2000arithmetic}, Moriwaki defined an arithmetic height function over finitely generated fields over $\Q$, which possesses various good properties, such as the Northcott property for big polarizations \cite[Theorem A]{moriwaki2000arithmetic}. We refer the reader to \textit{loc. cit.} for the precise definition. Thus we propose to replace the geometric height with the Moriwaki height.

For a field $\mathbf{k}$ of characteristic $0$, we define the arithmetic degrees by replacing $h$ in \eqref{eq_arith_deg} with a Moriwaki height over some chosen subfield $K\subset\mathbf{k}$ which is finitely generated over $\Q$. To be more specific, we take such a field $K$ over which $X$, $f$ and $x$ are defined. Moreover, we will prove that the arithmetic degrees are intrinsic, i.e. they are independent of \begin{itemize}
    \item the choice of the field $K$ of definition (Proposition \ref{prop_arith_deg_intr}),
    \item the arithmetic structure of $K$, i.e. an arithmetic variety whose function field is $K$, and a big and nef polarization on it, which are introduced by Moriwaki (Remark \ref{rema_ar_deg}). 
\end{itemize} 
This shows that the arithmetic degree depends only on the dynamical system. In this article, we generalize results over $\ov\Q$ in this setting. To summarize, we obtain the following:  
\begin{theo}[Proposition \ref{prop_arith_deg_intr}, Theorem \ref{theo_fundamental_ineq} and Proposition \ref{theo_KSC_C}]\label{thm_intro_ineq}
    For any $x\in X_f(\mathbf{k})$, $\ov\alpha(f,x)$ and $\underline\alpha(f,x)$ are well-defined. We have the fundamental inequality 
    \[\ov\alpha(f,x)\leq \lambda_1(f).\] In particular, if $f$ is polarized, i.e. $f^*L=\lambda_1(f)L$ for some ample line bundle $L$, then $\alpha(f,x)=\lambda_1(f)$ if the $f$-orbit of $x$ is Zariski-dense. 
\end{theo} 
Previous proofs of the Kawaguchi--Silverman conjecture primarily rely on the height machinery and the Northcott property; see \cite{matsuzawa2023recent} for details. We believe that these results can be recovered within our framework. Furthermore, the previously known counterexamples to the Kawaguchi–Silverman conjecture over function fields no longer serve as counterexamples in our setting.

In \cite{DGH2022high,song2023high}, an arithmetic analogue of the dynamical degree is introduced under the framework of Arakelov geometry, defined similarly to \eqref{eq_dyn_deg} using arithmetic line bundles and arithmetic intersection numbers instead. We generalize this notion for dynamical systems over finitely generated fields over $\mathbb{Q}$; refer to Definition \ref{def_ar_dyn_deg} for more details. In this case, we exploit the full power of the high codimensional Yuan's inequality in \cite[Theorem 1.3]{song2023high} to show the following:
\begin{theo}[Definition-Proposition \ref{coro_dyn_arith_deg} and Theorem \ref{rel_deg_fgf}]\label{theo_intro_reladeg}
    For a normal projective variety $X$ and a dominant rational self-map $f:X \dashrightarrow X$ defined over a finitely generated field $K$ over $\mb{Q}$, the $k$-th \textit{arithmetic dynamical degree} $\widehat \lambda_{k}(f)$ for $k\in \mb{Z}_+$ is well-defined, and $$\widehat{\lambda}_k(f)=\max\{\lambda_k(f),\lambda_{k-1}(f)\}.$$
\end{theo}
Note that Theorem \ref{thm_intro_ineq} follows directly from this formula.

Another significant advantage of our approach is the applicability to "transcendental points". We focus on the case that $\mathbf{k}$ is a universal domain, i.e. $\mathbf{k}$ is an algebraically closed field over $\Q$, and of infinite transcendental degree (for example, $\mathbf{k}=\C$). 

The following definition of transcendental points follows \cite[Section 4.1]{XY2023partial}.

\begin{defi}\label{def_trans}
Let $X_K$ be a variety over a subfield $K\subset \mathbf{k}$ and $X=X_K\times_{\Spe K}\Spe\mathbf{k}$. Here $X$ does not need to be irreducible. Denote $\pi_K: X\ra X_K$ the canonical projection. For a point $x\in X(\mathbf{k})$, denote $Z(x)_K$ the Zariski closure of $\pi_K(x)$ in $X_{K}$ and $Z(x)=\pi_K^{-1}(Z(x)_K)$. We call $\dim Z(x)$ the \emph{transcendence degree} of $x$ with respect to $X_K\ra \Spe K$. We say $x$ is \emph{transcendental} if $\dim Z(x)\geq 1$ and call it a \emph{purely transcendental point} if $\dim Z(x)=\dim X$.    
\end{defi}

We expect that purely transcendental points exhibit the maximal arithmetic complexity. We prove the following theorem.
\begin{theo}[Corollary \ref{coro_arithdeg_constant} and Theorem \ref{theo_verygeneral}]
Let $X$ be a normal projective variety and $f: X\dra X$ be a dominant rational self-map over a universal domain $\mathbf{k}$. Suppose $(X,f)$ is defined over a finitely generated field $K\subset \mathbf{k}$ over $\mb{Q}$. For $x,y\in X_f(\mathbf{k})$ satisfying $Z(x)_{K}=Z(y)_{K}$, we have 
\[\ov{\alpha}(f,x)=\ov{\alpha}(f,y) \text{ and } \underline{\alpha}(f,x)=\underline{\alpha}(f,y).\]

In particular, there exists a countable set $S$ of proper subvarieties of $X$, such that on $\in X(\mathbf{k})\setminus \bigcup\limits_{V\in S}V(\mathbf{k})$, the arithmetic degrees $\ov{\alpha}(f,\cdot)$ and $\underline{\alpha}(f,\cdot)$ are constant functions. Furthermore, if $f$ is an endomorphism, then $\alpha(f,x)=\lambda_1(f)$ for all $x\in X(\mathbf{k})\setminus \bigcup\limits_{V\in S}V(\mathbf{k})$.
\end{theo}

\subsection*{Comparison with previous works}
This paper is primarily inspired by \cite{Matsuzawa2018zero,ohnishi2022arakelov}. In \cite{Matsuzawa2018zero}, the arithmetic degrees are defined via geometric heights over function fields of characteristic $0$. The Kawaguchi--Silverman conjecture is not valid even for the case of endomorphism of varieties under the default of Northcott property. This is partly addressed in \cite{ohnishi2022arakelov}, by considering dynamical systems over adelic curves introduced in \cite{adelic} that admit the Northcott property. Note that finitely generated fields over $\Q$, equipped with the Moriwaki heights, which we are mainly concerned about, indeed fall into this category. Compared to \cite{ohnishi2022arakelov}, we show that the arithmetic degree does not depend on the choice of the field of definition. Moreover, instead of only considering algebraic points, we are able to deal with transcendental points by taking larger fields of definition. 

The main techniques of the paper come from \cite{moriwaki2000arithmetic} and \cite{song2023high}, which are mainly built on the arithmetic intersection theory.
\subsection*{Organization of the paper}
In Section \ref{Sec_height}, we recall the definition of Moriwaki heights over finitely generated fields. Moreover, we prove some fundamental properties which will be used to prove the intrinsicness of arithmetic degrees.

In Section \ref{Sec_arithdyndeg}, we give the definition of the arithmetic dynamical degree $\widehat{\lambda}_k(f)$ and the arithmetic degree $\alpha(f,x)$ for $(X,f)$ defined over a finitely generated field $K$. We prove the relative degree formula Theorem \ref{theo_intro_reladeg}
and establish the fundamental inequality $\ov\alpha(f,x)\leq \lambda_1(f)$.

In Section \ref{Sec_char0}, we discuss $(X,f)$ defined over an arbitrary field $\mathbf{k}$ of characteristic $0$. By taking a spread-out, we define the arithmetic degrees of a $\mathbf{k}$-point. Moreover, we are able to show that it is independent of the choice of spread-out. Further, we demonstrate that the Northcott property of the Moriwaki height ensures that certain results of the Kawaguchi--Silverman conjecture hold in this setting. In the case that $\mathbf{k}$ is a universal domain, we show that the arithmetic degree is invariant under the action of $\mr{Gal}(\mathbf k/K)$, where $K\subset \mathbf k$ is a finite generated subfield. Moreover, we show that very general $\mathbf k$-points of $X$ attain the maximal arithmetic complexity.

\section*{Acknowledgement}
We are very grateful to Junyi Xie, who provided the idea of characterizing arithmetic degrees of transcendental points. We would like to thank Shu Kawaguchi, Atsushi Moriwaki, and Yohsuke Matsuzawa for their valuable communications and advice.
In particular, we are thankful to Binggang Qu, one of whose questions proposed in a seminar motivated us to initiate this program. We would like to thank the referees for helpful comments that corrected mistakes and improved the exposition.

\section{Heights over finitely generated fields}\label{Sec_height}
\subsection{Hermitian line bundles on arithmetic varieties}\label{Sec_height_1}
We say an integral scheme $\scrX$ is an \textit{arithmetic variety} if it is flat, separated, and of finite type over $\Spe\mb{Z}$. We say $\scrX$ is \textit{projective} if the morphism $\scrX\rightarrow \Spe \Z$ is projective. Let $\shfL$ be a line bundle on $\scrX$, and $\lVert\cdot\rVert$ be a metric on $\shfL_{\C}^{\mr{an}}$. The pair $\ovl{\shfL}:=(\shfL,\lVert\cdot\rVert)$ is called a \textit{hermitian line bundle} if $\lVert\cdot\rVert$ is smooth and invariant under complex conjugation.

Let $d$ be the relative dimension of $\scrX\rightarrow\Spe \Z$. Let $\ovl\shfL_1, \ovl \shfL_2, \dots, \ovl\shfL_{d+1}$ be $d+1$ hermitian line bundles on $\mc{X}$.
Then after taking a generic resolution of singularities $f:\mc{X}^{\prime}\rightarrow \mc{X}$, the intersection number is defined as \begin{equation}\label{eq_intersection}
    (\ovl \shfL_1\cdot \ovl \shfL_2\cdots\ovl\shfL_{d+1}):=\widehat c_1(f^*\ovl \shfL_1)\cdot \widehat c_1(f^*\ovl \shfL_2)\cdots\widehat c_1(f^*\ovl \shfL_{d+1})\in\R
\end{equation}
            where the latter is defined in \cite{Soule1990AIT}. This is independent of the choice of resolution $f$ as shown in \cite{zhangposvar}.

Following \cite[Corollary 4.8]{zhangposvar}, we say a hermitian line bundle $\ovl\shfL$ is \textit{ample} if
\begin{enumerate}
    \item[(a)] $\shfL_\Q$ is ample;
    \item[(b)] $\shfL$ is relatively nef and the curvature $c_1(\shfL_\C,\lVert\cdot\rVert)$ of $\ovl\shfL$ is semipositive on $\scrX(\C)$;
    \item[(c)] for any horizontal closed subvariety $\mathcal Y\subset\scrX$, i.e. $\mathcal Y$ is closed, irreducible and $\mathcal Y\rightarrow \Spe\Z$ is surjective, \begin{equation*}\label{ineq_arith_ample}
        (\ovl \shfL|_{\mathcal Y})^{\dim(\mathcal Y)}>0.
    \end{equation*}
\end{enumerate}
Note that if we denote by $\zeta_{\mr{abs}}(\ovl \shfL)$ the \textit{absolute minimum} \begin{equation}\label{eq_abs_min}
    \zeta_{\mathrm{abs}}(\ovl \shfL):=\inf_{\substack{Y\subset X_{\ovl{\Q}}\\ Y\text{ is closed}}}\frac{(\ovl\shfL|_{\Delta_Y})^{\dim Y+1}}{(\dim Y+1)\deg_{\shfL_{\Q}}(Y)},
\end{equation}
where $\Delta_Y$ is the Zariski closure of $Y$ in $\scrX$, then we can replace the condition (c) by $\zeta_{\mr{abs}}(\ovl\shfL)>0$, see \cite{zhangposvar}, and \cite[Theorem 2.10]{Ballay2024Okounkov} for an explicit version.
We say $\ovl \shfL$ is \textit{semipositive} if the condition (b) is satisfied. We say $\ovl \shfL$ is \textit{nef} if $\ovl\shfL$ is semipositive and $\zeta_{\mr{abs}}(\ovl \shfL)\geq 0$.

If there exists a global section $s\in H^0(\scrX,\shfL)$ such that $\lVert s\rVert\leq 1$ on $\scrX(\C)$, then $\ovl \shfL$ is said to be \textit{effective}. We say $\ovl \shfL$ is \textit{strictly effective} if moreover $\lVert s\rVert<1$.
We say $\ovl \shfL$ is \textit{big} if $\shfL_\Q$ is big and $n\ovl \shfL$ is strictly effective for some positive integer $n$.

For hermitian line bundles $\ovl \shfL_1,\cdots\ovl \shfL_k$, we denote $\ovl \shfL_1\cdots\ovl\shfL_k\geq_{\mathrm n} 0$ if for any nef hermitian line bundle $\ov{\mc{H}}_1,\dots,\ov{\mc{H}}_{d+1-k}$, \[\big(\ov{\mc{H}}_1\cdots\ov{\mc{H}}_{d+1-k}\cdot \ovl \shfL_1\cdots\ovl\shfL_k)\geq 0.\]
Here, the subscript "n" stands for "numerical".

The intersection number defined in \eqref{eq_intersection} can be extended to the so-called integrable adelic metrized line bundles, and see \cite{zhang1995small}.
\subsection{Moriwaki heights}\label{subsec_moriwaki}
We begin by recalling the definition and some fundamental properties of the Moriwaki height, as established in \cite{moriwaki2000arithmetic}.

Let $K$ be a finitely generated field over $\Q$ with transcendence degree $e$. Let $B$ be a projective and normal arithmetic variety such that the function field of $B$ is $K$, which implies that $\dim B=e+1$. We fix a collection $(\ovl H_1,\dots, \ovl H_{e})$ of nef hermitian line bundles on $B$. The data $\ovl B=(B;\ovl H_1,\dots, \ovl H_{e})$ is referred as a \textit{polarization} or a \textit{model} of $K$. If $\ovl H=\ovl H_1=\dots=\ovl H_e$, we simplify the notation and write $\ovl B=(B;\ovl H)$. Following \cite{moriwaki2000arithmetic}, we say the polarization is ample (resp. nef, or big) if $\ovl H_i$ are ample (resp. nef, or big) hermitian line bundles. In the following, we assume that every polarization is taken to be big and nef unless otherwise specified. We denote by $B^{(1)}$ the set of schematic points of codimension $1$ in $B$, i.e. $$B^{(1)}:=\{\Gamma\in B\mid \mr{codim}_B(\Gamma)=1\}.$$

Let $X$ be a projective $K$-variety of dimension $d$, and $L$ be a line bundle on $X$. We can take a $B$-model $(\mc{X},\ovl\shfL)$ of $(X,L)$, that is, $\mc{X}$ is a $B$-model of $X$, and $\ovl\shfL=(\shfL,\nm)$ is a hermitian line bundle over $\mc{X}$ such that $\shfL_K\simeq L$. For any $x\in X(\ovl K)$, we denote by $\Delta_x$ the Zariski closure of the image of $x:\Spe(\ovl K)\rightarrow X\rightarrow \mc{X}$ in $\mc{X}$. The Moriwaki height $h^{\ovl B}_{(\mc{X},\ov{\mc{L}})}$ is given by the arithmetic intersection theory:
$$h^{\ovl B}_{(\mc{X},\ov{\mc{L}} )}(x)=\frac{1}{[K(x):K]}(\ovl \shfL|_{\Delta_x}\cdot \pi^*\ovl H_1|_{\Delta_x}\cdots \pi^*\ovl H_e|_{\Delta_x}).$$

If $X=\mathbb{P}^n_K$, we have the \textit{naive height}:
\begin{align*}
    h_{\mathrm{nv}}^{\ovl B}([\lambda_0:\dots:\lambda_n])&=\sum_{\Gamma \in B^{(1)}}\max_i\{-\mathrm{ord}_\Gamma(\lambda_i)\}(\ovl H_1|_\Gamma\cdots \ovl H_e|_\Gamma)\\
    &\kern 5em+\int_{B(\C)}\log(\max\{\lvert\lambda_i\rvert\})c_1(\ovl H_1)\dots c_1(\ovl H_e),
\end{align*}
where $[\lambda_0:\dots:\lambda_n]\in \mathbb{P}^n(K)$. Here if $\Gamma$ is vertical, i.e. $\Gamma$ is contained in a fiber $B_p$ of $B$ over a prime number $p$, then $(\ovl H_1|_\Gamma\cdots\ovl H_e|_\Gamma)$ is simply the usual intersection number on an algebraic variety over the finite field $\mathbb F_p$. The notation $\lvert\lambda_i\rvert$ denotes the function $(x\in B(\C))\mapsto \lvert \lambda_i(x)\rvert$ where $\lvert\cdot\rvert$ is the usual absolute value on $\C$. 

Let $\scrX=\mathbb P^n_B$, $\O_{\mathbb{P}^n_B}(1)$ be the tautological line bundle on $\mathbb P^n_B$, and $\nm_{\infty}$ be the metric on $\O_{\mathbb{P}^n_B}(1)$ given by
\begin{equation}\label{eq_infinity_fubini}
    \lVert X_i\rVert_\infty:=\frac{\lvert X_i\rvert}{\max_j\{\lvert X_j\rvert\}}.
\end{equation} 
Then the naive height $h^{\ovl B}_{\mathrm{nv}}$ coincides with $h^{\ovl B}_{\mathbb P^n_B,(\O_{\mathbb{P}^n_B}(1),\nm_\infty)}$.
In this paper, the Northcott property of Moriwaki heights is essential. We recall the theorem here.
\begin{theo}[cf. \cite{moriwaki2000arithmetic}, Theorem 4.3]
    Let $\ov{\mc{L}}$ be a hermitian line bundle on $\mc{X}$ whose generic fiber $L=\mc{L}_K$ is ample, then for any $C\in\R$ and $\delta\geq 1$,
    the set $$\{x\in X(\ovl K)\mid h^{\ov B}_{(\scrX,\ovl\shfL)}(x)\leq C,[K(x): K]\leq \delta\}$$
is finite. 
\end{theo} 
We are going to prove some results related to the intrinsic of arithmetic degrees. We begin with the definition of comparability of height functions with respect to a field extension.
\begin{defi}\label{defi_comp}
    Let $K'/K$ be an extension of finitely generated fields over $\Q$ such that $\mathrm{tr.deg}(K/\Q)=e$ and $\mr{tr.deg}(K'/\Q)=e'$. Let $\ovl B=(B;\ovl H_1,\dots,\ovl H_e)$ and $\ovl B'=(B';\ovl H_1',\dots,\ovl H'_{e'})$ be models with big and nef polarizations of $K$ and $K'$ respectively. Let $X$ be a projective $K$-variety. Let $L$ be a line bundle on $X$ and $L'$ be a line bundle on $X_{K'}$ (note that $L'$ is not necessarily the extension $L_{K'}$ of $L$). Let $(\scrX,\ovl \shfL)$  (\textit{resp. }$(\mc{X}',\ovl \shfL')$) be a $B$-model(\textit{resp. }$B'$-model) of $(X,L)$(\textit{resp. }$(X_{K'},L')$). We say $h^{\ovl B}_{(\scrX,\ovl \shfL)}$ and $h^{\ovl B'}_{(\scrX',\ovl \shfL')}$ are \textit{comparable} if 
    there exists constants $C_1,C_2>0$ and $D_1,D_2\in\R$ such that for any $x\in X(\ovl K)$, $$C_1h_{(\scrX,\ovl\shfL)}^{\ovl B}(x)-D_2\leq h_{(\scrX',\ovl \shfL')}^{\ovl B'}(x_{K'})\leq C_2 h_{(\scrX,\ovl \shfL)}^{\ovl B}(x)+D_1,$$
    where $x_{K'}:\Spe\ovl K'\rightarrow X_{K'}$ is the point obtained from the base change of $x:\Spe\ovl K\rightarrow X,$ and $C_1,D_1,C_2,D_2$ are independent of the choice of $x$.
\end{defi}
The following lemma is from the definition directly.
\begin{lemm}\label{lemm_comp}
    We keep the notations and assumptions as in Definition \ref{defi_comp}. Let $K''$ be a finitely generated extension of $K'$ and $\mr{tr.deg}(K''/\Q)=e''$. Let $\ovl B''=(B'';\ovl H_1'',\dots,\ovl H_{e''}'')$ be a big and nef polarization of $K''$. Let $L''$ be a line bundel on $X_{K''}$. Let $(\scrX'',\ovl \shfL'')$ be a $B''$-model of $(X_{K''},L'')$. Consider the following three conditions: \begin{enumerate}
        \item[\textnormal{(a)}] $h^{\ovl B}_{(\scrX,\ovl \shfL)}$ and $h^{\ovl B'}_{(\scrX',\ovl \shfL')}$ are comparable,
        \item[\textnormal{(b)}] $h^{\ovl B'}_{(\scrX',\ovl \shfL')}$ and $h^{\ovl B''}_{(\scrX'',\ovl \shfL'')}$ are comparable,
        \item[\textnormal{(c)}] $h^{\ovl B}_{(\scrX,\ovl \shfL)}$ and $h^{\ovl B''}_{(\scrX'',\ovl \shfL'')}$ are comparable.
    \end{enumerate}
    If \textnormal{(a)} and \textnormal{(b)} hold, so is \textnormal{(c)}. Conversely, if \textnormal{(c)} and \textnormal{(b)} hold, so is \textnormal{(a)}.
\end{lemm}

\begin{prop}\label{prop_compar}
Let $\ovl B=(B;\ovl H_1,\dots,\ovl H_e)$ and $\ovl B'=(B';\ovl H_1',\dots,\ovl H'_{e})$ be two models of $K$. Assume that $\ovl H_i,\ovl H_i'$ are big and nef hermitian line bundles.
Let $X$ be a projective $K$-variety, $L$ and $L'$ be ample line bundles on $X$, and $(\scrX,\ovl \shfL)$ (\textit{resp. }$(\mc{X}',\ovl \shfL')$) be a $B$-model(\textit{resp. }$B'$-model) of $(X,L)$(\textit{resp. }$(X,L')$). 
    Then $h_{(\scrX,\ovl\shfL)}^{\ovl B}$ and $h_{(\scrX',\ovl \shfL')}^{\ovl B'}$ are comparable.
\end{prop}
\begin{proof}
Let $B''$ be the normalization of the graph of the birational map $B\dashrightarrow B'$. Let $g:B''\rightarrow B$ and $g':B''\rightarrow B'$ be the natural projections. By the projection formula and the birational invariance of the arithmetic volume \cite[Theorem 4.2]{moriwaki2009convol}, $g^* \ovl H_i, g'^*\ovl H_i'(i=1,2,\dots, e)$ remain nef and big. 
By the Yuan's inequality \cite[Theorem 5.2.2]{yuan2021adelic}, there exists $C\in\N_+$ such that for each $i=1,2,\dots,e$, the differences $C g'^*\ovl H_i'-g^*\ovl H_i$ and $C g^*\ovl H_i-g'^*\ovl H_i'$ are effective.
Indeed, $$C^{\dim B}\ovl H_i^{\dim B}-\dim B\cdot C^{\dim B-1}g^*\ovl H_i^{\dim B-1}g'^*\ovl H_i'>0$$ for $C\gg0.$
Let $(\scrX_{B''}, \ovl \shfL_{B''})$ and $(\scrX'_{B''}, \ovl \shfL'_{B''})$ denote the base changes of $(\scrX,\ovl \shfL)$ and $(\scrX',\ovl \shfL')$ respectively.
Then by \cite[Proposition 3.3.7(5)]{moriwaki2000arithmetic}, we obtain the following inequalities:
\[h_{(\scrX,\ovl\shfL)}^{\ovl B}=h_{(\scrX_{B''},\ov\shfL_{B''})}^{(B'';g^*\ov H_1,\dots,g^*\ov H_e)}\leq h_{(\scrX_{B''},\ov\shfL_{B''})}^{(B'';Cg'^*\ov H'_1,\dots,Cg'^*\ov H'_e)}+O(1)=C^{e}h_{(\scrX_{B''},\ov\shfL_{B''})}^{(B'';g'^*\ov H'_1,\dots,g'^*\ov H'_e)}+O(1),\]
and similarly,
\[h_{(\scrX_{B''},\ov\shfL_{B''})}^{(B'';g'^*\ov H'_1,\dots,g'^*\ov H'_e)}\leq h_{(\scrX_{B''},\ov\shfL_{B''})}^{(B'';Cg^*\ov H_1,\dots,Cg^*\ov H_e)}+O(1)=C^{e}h_{(\scrX,\ovl\shfL)}^{\ovl B}+O(1).\]

On the other hand, since $L$ and $L'$ are both ample, Lemma \ref{lemm_comp} implies that the heights $h_{(\scrX_{B''},\ov\shfL_{B''})}^{(B'';g'^*\ov H'_1,\dots,g'^*\ov H'_e)}$ and $h_{(\scrX'_{B''},\ov\shfL'_{B''})}^{(B'';g'^*\ov H'_1,\dots,g'^*\ov H'_e)}=h_{(\scrX',\ovl \shfL')}^{\ovl B'}$ are comparable, which concludes the proof.

\end{proof}

The following result is the key ingredient in constructing arithmetic degrees over fields of characteristic zero. As we will see, it ensures that the arithmetic degrees of points in an algebraic dynamical system over any algebraically closed field of characteristic zero are well-defined. (Proposition \ref{prop_arith_deg_intr}). We need to consider the following situation. Let $K$ be a finitely generated field over $\Q$, and $X$ be a projective $K$-variety. We fix a big and nef polarization $\ovl B=(B,\ovl H)$ of $K$. Suppose that we have two field extensions $K_1/K$ and $K_2/K$ where both $K_1$ and $K_2$ are finitely generated fields over $\Q$, then the Moriwaki heights on $X_{K_1}(\ovl K_1)$ and $X_{K_2}(\ovl K_2)$, when restricted to $X(\ovl K)$, should be comparable. To be more precise, let $x\in X(\ovl K)$. Consider the following diagram
\[\begin{tikzcd}
 & \mr{Spec}(\ovl K_1)\arrow{r}{x_{K_1}} \arrow{d} &X_{K_1}\arrow{d}\\
 & \mr{Spec}(\ovl K) \arrow{r}{x} &X\\
 & \mr{Spec}(\ovl K_2)\arrow{r}{x_{K_2}} \arrow{u} &X_{K_1}\arrow{u},
\end{tikzcd}
\]
where $x_{K_i}\in X_{K_i}(\ovl K_i)$ are obtained by base changes for $i=1,2.$ Now we take Moriwaki heights $h_{\scrX_i,\ovl{\mathcal A}_i}^{\ovl B_i}$ on $X_{K_i}(\ovl K_i)$ that is, for each $i=1,2$, $\ovl B_i=(B_i, \ovl H_i)$ is a model of $K_i$, $\scrX_i$ is a $B_i$-model of $X_{K_i}$, $\ovl{\mathcal A_i}=(\mc{A}_i,\nm_i)$ is a hermitian line bundle on $\scrX_i$ such that $(A_i)_{K_i}$ is ample.
For the remainder of this subsection, we keep those notations and assumptions.

\begin{theo}\label{theo_height_comparable}
    There exists constants $C_1,C_2>0$ and $D_1,D_2\in \R$ such that for any $x\in X(\ovl K)$, $$C_1h_{(\scrX_1,\ovl{\mathcal A}_1)}^{\ovl B_1}(x_{K_1})-D_1\leq h_{(\scrX_2,\ovl {\mathcal A}_2)}^{\ovl B_2}(x_{K_2})\leq C_2 h_{(\scrX_1,\ovl {\mathcal A}_1)}^{\ovl B_1}(x_{K_1})+D_2,$$
    where $C_1,D_1,C_2,D_2$ are independent of the choice of $x\in X(\ov{K})$.
\end{theo}

Before the proof, we prepare some useful lemmas.
\begin{lemm}\label{lemm_naive}
    Assume that $B_2=B$, $B_1=\mathbb{P}_\Z^1\times B$, and let $\ovl H_1=p^*(\mathcal{O}(1),\nm_{\mathrm{FS}})\ot f^*\ovl H$, where $\ovl H$ is a big and nef hermitian line bundle over $B$ and $p:B_1\rightarrow \mathbb{P}_\Z^1$, $f:B_1\rightarrow B$ are canonical projections. The Fubini-Study metric $\nm_{\mr{FS}}$ on $\mc{O}(1):=\mc{O}_{\mathbb{P}_{\Z}^1}(1)$ is given by 
    $$\lVert X_i\rVert_{\mr{FS}}:=\frac{\lvert X_i\rvert}{\sqrt{\lvert X_0\rvert^2+\lvert X_1\rvert^2}}\text{ }(i=0,1),$$
    where $|
    \cdot|$ denotes the usual absolute value. Then there exists a constant $C>0$, independent of the choice of $x\in \mathbb P^n(K)$ such that 
    $$(e+1)h^{\ovl B}_{\mathrm{nv}}(x)\leq h_{\mathrm{nv}}^{\ovl B_1}(x_{K_1})\leq C h^{\ovl B}_{\mathrm{nv}}(x)$$
    where $K_1=K(t)$ for an indeterminate $t$ and $e=\mr{tr.deg}(K/\mb{Q})$ is the transcendence degree of $K$ over $\mb{Q}$.
\end{lemm}
\begin{proof}
    Consider $x=[\lambda_0:\dots:\lambda_n]\in \mathbb P^n(K)$ where each $\lambda_i\in K$. Then \begin{equation}\label{eq_naiveheight_comp}
    h^{\ovl B_1}_{\mathrm{nv}}(x_{K_1})=\sum_{\Gamma \in B_1^{(1)}}\max_i\{-\mathrm{ord}_\Gamma(\lambda_i)\}(\ovl H_1|_\Gamma)^{1+e}+\int_{B_1(\C)}\log(\max_i\{\lvert \lambda_i\rvert\})c_1(\ovl H_1)^{1+e}.
\end{equation}

 In $B_1^{(1)}$, any prime divisor can be classified into the following two types:
 \begin{itemize}
     \item Type 1: a divisor of form $f^{-1}\Gamma=\mathbb{P}^1_\Z\times\Gamma$, where $\Gamma\in B^{(1)}$.
     \item Type 2: a divisor of form $\ov{\{y\}}$, that is, the Zariski closure of a closed point $y\in \mb{P}^1_K\hookrightarrow \mb{P}_{\mb{Z}}^1\times B$.
 \end{itemize}  For all $\Gamma$ of type 2, we have $\max\limits_i\{-\mathrm{ord}_{\Gamma}(\lambda_i)\}=0$ since $\lambda_i\in K, i=0,\dots,n$. Therefore, it suffices to consider the prime divisors of type 1 in $B_1$, given by $f^{-1}\Gamma=\mathbb{P}^1_\mathbb{Z}\times\Gamma$, where $\Gamma\in B^{(1)}$. Consequently, we obtain
\begin{equation}\label{eq_part1}
\sum_{\Gamma \in B_1^{(1)}}\max_i\{-\mathrm{ord}_\Gamma(\lambda_i)\}(\ovl H_1|_\Gamma)^{1+e}=\sum_{\Gamma\in B^{(1)}}\max_i\{-\mathrm{ord}_{\Gamma}(\lambda_i)\}(\ovl H_1|_{f^{*}\Gamma})^{1+e}.
\end{equation}
If $\Gamma$ is horizontal in $B$, then
\[(\ovl H_1|_{f^*\Gamma})^{1+e}=(\ovl H_1|_{f^{-1}\Gamma})^{1+e}= \frac{(e+1)e}{2}(\mathcal{O}(1),\lVert\cdot\rVert_{\mr{FS}})^2(H|_{\Gamma_\Q})^{e-1}+(e+1)(\ovl H|_\Gamma)^{e}.\]

Recall that the absolute minimum $\zeta_{\mathrm{abs}}(\ovl H)>0$ (see \eqref{eq_abs_min} for definition) since $\ov H$ is ample \cite[Theorem 2.10]{Ballay2024Okounkov}.
Hence $$\frac{(\ovl H|_\Gamma)^e}{e(H|_{\Gamma_\Q})^{e-1}}\geq \zeta_{\mathrm{abs}}(\ovl H)> 0.$$
We have the inequality \begin{equation}\label{ineq_deg_arith_geom}
    (e+1)(\ovl H|_\Gamma)^{e}\leq (\ovl H_1|_{f^{-1}\Gamma})^{1+e}\leq C\cdot(\ovl H|_\Gamma)^e,\end{equation}
where the constant $C$ is given by 
\[ C=(e+1)+\frac{(e+1)e}{2e\zeta_{\mathrm{abs}}(\ovl H)}(\mathcal{O}(1),\lVert\cdot\rVert_{\mr{FS}})^2=(e+1)\left(1+\frac{1}{4\zeta_{\mathrm{abs}}(\ovl H)}\right).\]
Here, the equality $(\mathcal{O}(1),\lVert\cdot\rVert_{\mr{FS}})^2=\frac{1}{2}$ follows from \cite[Page 667]{Kohler2002flag}.

If $\Gamma$ is vertical, then $(\ovl H_1|_{f^*\Gamma})^{1+e}=(1+e)(\ovl H|_{\Gamma})^{e}$ by the projection formula of algebraic cycles. 

For the integration term in \eqref{eq_naiveheight_comp}, we have
\begin{equation}\label{eq_part2}
\begin{aligned}
    \int_{B_1(\C)}\log(\max_i\{\lvert \lambda_i\rvert\})c_1(\ovl H_1)^{1+e}&=(e+1)\int_{B_1(\C)}\log(\max_i\{\lvert\lambda_i\rvert\})p^*c_1(\ov{\mathcal O(1)})f^*c_1(\ovl H)^{e}\\
    &=(e+1)\int_{B(\C)}\log(\max_i\{\lvert\lambda_i\rvert\})c_1(\ovl H)^{e},
\end{aligned}
\end{equation}
where the second equality follows from Fubini's theorem.
We may assume that $\lambda_0=1$. Then, by \eqref{eq_naiveheight_comp},\eqref{eq_part1},\eqref{ineq_deg_arith_geom} and \eqref{eq_part2}, we obtain
$$(e+1)h^{\ovl B}_{\mathrm{nv}}(x)\leq h_{\mathrm{nv}}^{\ovl B_1}(x_{K_1})\leq C h^{\ovl B}_{\mathrm{nv}}(x).$$
\end{proof}

\begin{lemm}\label{lemm_reduction}
    We keep the same assumption as in Lemma \ref{lemm_naive}. Let $\mc X$ be a $B$-model of $X$, and $\ovl{\mc A}=(\mc A,\lVert\cdot\rVert)$ be a hermitian line bundle on $\mc X$ such that $\mc A_K$ is ample. 
    Then 
    $h_{(\scrX_1,\ovl{\mathcal A}_1)}^{(B_1,\ovl H_1)}$ and $ h_{(\scrX,\ovl {\mc A})}^{(B,\ovl H)}$ are comparable.
\end{lemm}
\begin{proof}
    First, observe that we may assume that $x\in X(K)$. Indeed, since $K_1/K$ is purely transcendental, we have $[K_1(x_{K_1}):K_1]=[K(x):K]$. By combining this with \cite[Proposition 3.3.1]{moriwaki2000arithmetic}, we may assume without loss of generality that $K(x)=K$ by replacing $B$ with its normalization in $K(x)$.
    
    We may fix a closed embedding $X\xrightarrow{\iota} \mathbb{P}^{n}_K$. Let $\scrX$ be the Zariski closure of $X$ in $\mathbb{P}^n_B$ with respect to $\iota$. By abuse of notation, we may still denote by $\iota $ the embedding $\scrX\rightarrow \mathbb{P}^n_B$. Define $\ovl{\shfL}:=\iota^*(O_{\mathbb{P}^n_B}(1),\lVert\cdot\rVert_\infty)$, where $O_{\mathbb{P}^n_B}(1)$ is the tautological bundle over $\mathbb{P}^n_B$ and $\lVert\cdot\rVert_\infty$ is the quotient metric given by \eqref{eq_infinity_fubini}. By \cite[Proposition 3.3.2]{moriwaki2000arithmetic}, we have $h^{\ovl B}_{(\widetilde\scrX,\ovl{\shfL})}=h^{\ovl B}_{\mathrm{nv}}|_{X(\ovl K)}$.
    Therefore, due to Proposition \ref{prop_compar}, we may replace $\ovl{\mc{A}}_2$ by $\ovl\shfL$. Applying the same reasoning to $\iota_{K_1}:X_{K_1}\rightarrow \mathbb{P}^n_{K_1}$, we conclude the proof from Lemma \ref{lemm_naive}.
\end{proof}

\begin{proof}[Proof of Theorem \ref{theo_height_comparable}]
    We take $B'=B_1\times_{\Spe\Z} B_2$ and fix an ample hermitian line bundle $\ovl H'$ over $B'$. We take an ample hermitian line bundle $\ovl{\mc{A}}'$ over $\scrX':=\scrX_1\times_{B_1}B'$. So it suffices to prove that $h_{(\scrX_i,\ovl{\mathcal A}_i)}^{\ovl B_i}$ and $h_{(\scrX', \ovl {\mathcal A}')}^{\ovl B'}$ are comparable for $i=1,2$ by using Lemma \ref{lemm_comp}. In the following, we may assume that $B_2=B$ (which implies that $K_2=K$) and there exists a morphism $f: B_1\rightarrow B$.
    
    Assuming that $\mathrm{tr.deg}(K_1/K)=d$, we can take $t_1,\dots,t_d\in K_1$ algebraically independent over $K$. Then $K(t_1,\dots,t_d)/K$ is purely transcendental and $[K_1:K(t_1,\dots, t_d)]<\infty$. By \cite[Proposition 3.3.1]{moriwaki2000arithmetic}, we may assume that $K_1=K(t_1,\dots,t_d)$. Applying Lemma \ref{lemm_reduction} inductively for $d$-times, we can see that the theorem holds in the case that $B_1=(\mathbb{P}^1_\Z)^d\times B$ and $\ovl H_1=\mathop{\bigotimes}\limits_{i}p_i^*(\mathcal O(1),\nm_{\mathrm{FS}})\ot f^*\ovl H$ where $p_i: B_1\rightarrow \mathbb{P}^1_\Z$ is the $i$-th projection. Then we are done by Proposition \ref{prop_compar}.
\end{proof}

\begin{rema}
    By arithmetic intersection theory, we can define the height of a closed subvariety. If we assume that both $X$ and $B_{\Q}$ are smooth, then the projection formula for arithmetic cycles \cite[Proposition 5.18]{moriwaki2014arakelov} gives a similar comparable result for the heights of closed subvarieties.
\end{rema}
\subsection{Height associated to divisors} 
In this subsection, we fix a model $\ovl B=(B;\ovl H_1,\dots, \ovl H_e)$ of a finitely generated field $K$.
Let $X$ be a projective normal variety. We denote by $\mathrm{Div}_\R(X)$ the group of $\R$-Cartier divisors. For any $D=a_1 D_1+\dots+a_r D_r\in \mathrm{Div}_\R(X)$, where $a_i\in \R$ and $D_i$ are Cartier integral divisors, a Moriwaki height associated to $D$ is of the form
$$a_1 h_{(\scrX_1,\ov\shfL_1)}^{\ovl B}+\dots+a_r h_{(\scrX_r,\ov\shfL_r)}^{\ovl B},$$
where $(\scrX_i,\ov{\shfL}_i)$ is a $B$-model of $(X,\O_X(D_i))$ for each $i=1,\dots, r$. For simplicity, we denote this height by $h^{\ovl B}_{\ovl D}(\cdot)$ or $h_{\ovl D}(\cdot)$ if there is no ambiguity, where $\ovl D$ represents the data $(a_1,\dots, a_r; (\scrX_1,\ov\shfL_1),\dots, (\scrX_r,\ov\shfL_r))$. 
\begin{prop}\label{prop_unbound_big}
    Let $D$ be a $\R$-Cartier big divisor. Then for $\delta\in\Z_+$ and $C\in\R$, the set $\{x\in X(\ovl K)\mid [K(x):K]\leq \delta,h_{\ovl D}(x)\leq C\}$ is not Zariski dense.
\end{prop}
\begin{proof}
    We can write $D=A+E$ where $A$ is ample and $E$ is effective. We may choose height functions $h_{\ovl A}(\cdot)$ and $h_{\ovl E}(\cdot)$ such that $h_{\ovl D}=h_{\ovl A}+h_{\ovl E}.$
    There exists a constant $C_0$ such that $h_{\ovl D}(x)\geq h_{\ovl A}(x)-C_0$ for all $(x\in X\setminus \mathrm{Supp}(E))(\ovl K)$. Therefore we have
    $$\begin{aligned}
        &\{x\in X(\ovl K)\mid [K(x):K]\leq \delta,h_{\ovl D}(x)\leq C\}\\
        &\subset \{x\in X(\ovl K)\mid [K(x):K]\leq \delta,h_{\ovl A}(x)\leq C+C_0\}\cup \mathrm{Supp}(E)(\ovl K),
    \end{aligned}$$
    which is not Zariski dense by the Northcott property \cite[Theorem A]{moriwaki2000arithmetic}.
\end{proof}
\begin{prop}\label{prop_pullback}
    Let $f: Y\rightarrow X$ be a morphism of projective normal $K$-varieties, which extends to a $B$-morphism $F:\mc{Y}\rightarrow \scrX$ where $\mc{Y}$ and $\scrX$ are $B$-models of $Y$ and $X$ respectively. Let $D=a_1 D_1+\dots+a_r D_r$ be a $\R$-Cartier divisor where $a_i\in \R$ and $D_i$ are Cartier integral divisors. We fix a Moriwaki height function $h_{\ovl D}=a_1h^{\ovl B}_{(\scrX,\ovl\shfL_1)}+\dots+a_r h^{\ovl B}_{(\scrX,\ovl\shfL_r)}$ where $(\scrX,\ovl\shfL_i)$ is a $B$-model of $(X,\O_X(D_i))$ for each $i$. Then $h_{\ovl D}\circ f$ is a Moriwaki height function associated with $f^*D$.
\end{prop}
\begin{proof}
    By the projection formula \cite[Proposition 5.18]{moriwaki2014arakelov}, we have $h_{\ovl D}\circ f=a_1h^{\ovl B}_{(\scrX,F^*\ovl\shfL_1)}+\dots+a_r h^{\ovl B}_{(\scrX,F^*\ovl\shfL_r)}.$
\end{proof}
\begin{prop}\label{prop_num_triv}
    Let $D$ be a numerically trivial $\R$-Cartier divisor, and $A$ be an ample $\R$-Cartier divisor. We fix height functions $h_{\ovl D}(\cdot)$ and $h_{\ovl A}(\cdot)$ associated to $D$ and $A$ respectively such that $h_{\ovl A}(x)\geq 1, \forall x\in X(\ov{K})$. Then there exists a constant $C$ such that 
    $$h_{\ovl D}(x)\leq C\sqrt{h_{\ovl A}(x)}$$
    for any $x\in X(\ovl K).$
\end{prop}
\begin{proof}
    We follow the same strategy as in \cite[Proposition B.3]{matsuzawa2017bounds}.
    We may assume that $D$ and $A$ are Cartier divisors. Since $mD$ is algebraically equivalent to $0$ for sufficiently large $m$, we may further assume that $D$ is algebraically trivial. Since $K$ is of characteristic $0$, after replacing $K$ with a finite extension, we have a morphism $\mr{Alb}:X\rightarrow \mr{Alb}(X)$ where $\mr{Alb}(X)$ is an Albanese variety of $X$, such that $D=\mr{Alb}^* E$ for a divisor $E$ on $\mr{Alb}(X)$ which is algebraically trivial by \cite[Theorem 9.5.4]{FGAexplained}. Then the rest of the proof follows \cite[Theorem B.5.9]{Hindry2000diophantine}.
\end{proof}

\section{Dynamics over finitely generated fields}\label{Sec_arithdyndeg}
In this section, let $X$ be a normal $d$-dimensional projective variety over a finitely generated field $K$ over $\mb{Q}$ and $f: X\dashrightarrow X$ be a dominant rational self-map. The indeterminacy locus $I(f)$ is defined as the complement of the largest open subset where $f$ is well-defined. Denote that
\[X_f(\ov{K}):=\{x\in X(\ov{K})\mid x\notin I(f^n), \forall n\geq 0\}.\]
\subsection{Arithmetic dynamical degrees} Let $(B,\ov{H})$ be a model of $K$ such that $\ov{H}$ is an ample hermitian line bundle on $B$. 
Let $A$ be an ample line bundle on $X$, and consider a $B$-model $(\mc{X},\ov{\mc{A}})$ of $(X,A)$. Furthermore, to guarantee that the arithmetic dynamical degree is positive, we must assume $\ov{\mc{A}}_1=\ov{\mc{A}}(-t)=(\mc{A},e^{t}\|\cdot\|)$ is a nef hermitian line bundle for some $t$, where $\|\cdot\|$ denotes the metric associated with $\ov{\mc{A}}$.

The dominant rational self-map $f: X\dra X$ can be lifted to a dominant rational self-map $F:\mc{X}\dra \mc{X}$. We take a resolution of indeterminacy 
\[
\begin{tikzcd}
 & \Gamma_F \arrow{dl}{\pi_1} \arrow{dr}{\pi_2}\\
\mc{X} \arrow[dashed]{rr}{F} && \mc{X} ,
\end{tikzcd}
\]
where $\Gamma_F$ is the normalization of the graph of $F$ in $\mc{X}\times_B\mc{X}$, $\pi_1$ is a birational morphism to the first coordinate, and $\pi_2$ is a generically finite morphism to the second coordinate. By a slight abuse of language, we will also use $\pi_1$ and $\pi_2$ to refer to the morphisms from $\Gamma_{F,\mathbb{Q}}$ to $\mathcal{X}_{\mathbb{Q}}$ induced by the projections. The arithmetic dynamical degree can be defined as follows.

\begin{defi}\label{def_ar_dyn_deg}
    For an integer $1\leq k\leq n$, we define
$$\ardeg_{k,\ov{\mc{A}}}(F):=\big(\pi_1^*\ovl{\mathcal A}^{d-k+1}\cdot\pi_2^*{\ovl {\mathcal A}}^{k}\cdot\pi_1^*\pi^*\ovl H^{e}\big),$$
where $e$ is the transcendence degree of the extension $K/\mb{Q}$ and $\pi:\mc{X}\ra B$ denotes the canonical map.

\end{defi}

Due to the choice of hermitian line bundle $\ov{\mc{A}}$, the degree $\ardeg_{k,\ov{\mc{A}}}(F)$ is positive. In fact, since $\ov{\mc{A}}=\ov{\mc{A}}_1(t)$ for some $t>0$, we have
\[
    \begin{aligned}\allowdisplaybreaks
\ardeg_{k,\ov{\mc{A}}}(F)=&\big(\pi_1^*\ovl{\mathcal A}^{d-k+1}\cdot\pi_2^*{\ovl {\mathcal A}}^{k}\cdot\pi_1^*\pi^*\ovl H^{e}\big)\\    
=&\big(\pi_1^*\ovl{\mathcal A}_1^{d-k+1}\cdot\pi_2^*{\ovl {\mathcal A}}_1^{k}\cdot\pi_1^*\pi^*\ovl H^{e}\big)+(d-k+1)t\big(\pi_1^* \mc{A}_{\mb{Q}}^{d-k}\cdot \pi_2^*\mc{A}_{\mb{Q}}^k\big)\big(H_{\mb{Q}}^e\big)\\
&\kern 14em+kt\big(\pi_1^* \mc{A}_{\mb{Q}}^{d-k+1}\cdot \pi_2^*\mc{A}_{\mb{Q}}^{k-1}\big)\big(H_{\mb{Q}}^e\big)>0.
\end{aligned}
\]

The degree function satisfies the submultiplicity formula below, which follows from the high codimensional Yuan's inequality as established in \cite[Theorem 2.5]{song2023high}.

\begin{theo}\label{theo_deg_submulti}
With all the notation introduced at the beginning of this section, there exists a constant $C>0$, depending only on $\mc{X}$ and $\ov{\mc{A}}$, such that for any dominant rational self-maps $F:\mc{X}\dra \mc{X}$ and $G:\mc{X}\dra \mc{X}$, we have
\[\ardeg_{k,\ov{\mc{A}}}(F\circ G)\leq C\cdot\ardeg_{k,\ov{\mc{A}}}(F)\cdot\ardeg_{k,\ov{\mc{A}}}(G).\]
\end{theo}
\begin{proof}
Consider the diagram. 
\begin{equation}\label{bigdiagram}
 \begin{tikzcd}
		&                                                   & \Gamma \arrow{ld}{u} \arrow{rd}{v}   &                                                   &                    \\
		& \Gamma_G \arrow{ld}{\pi_1} \arrow{rd}{\pi_2} &                                           & \Gamma_F \arrow{ld}{\pi_3} \arrow{rd}{\pi_4} &                    \\
		\mc{X} \arrow{rrd}[swap]{\pi} \arrow[dashed]{rr}[swap]{G} &                                                   & \mc{X} \arrow{d}{\pi} \arrow[dashed]{rr}[swap]{F} &                                                   & \mc{X} \arrow{lld}{\pi} \\
		&                                                   & B                              &                                                   &                   
	\end{tikzcd}   
\end{equation}
Here, $\Gamma$ is the normalization of the graph of $\pi_3^{-1}\circ\pi_2$ in $\Gamma_G\times_B \Gamma_F$, while $u$ and $v$ denote the morphisms to the first and second coordinates, respectively.

 Denote $\pi\circ \pi_1\circ u$ by $\wt{\pi}$. The intersection number
 \[\begin{aligned}
      \big(u^*\pi_2^*\ov{\mc{A}}^{d+1}\cdot \wt{\pi}^*\ov{H}^e\big)&=\mr{topdeg}(\pi_2)\cdot\big(\ov{\mc{A}}^{d+1}\cdot \pi^*\ov{H}^e\big)\\
      &=\mr{topdeg}(\pi_2)\left(\big(\ov{\mc{A}}_1^{d+1}\cdot \pi^*\ov{H}^e\big)+(d+1)t(A^{d+1}\cdot \pi^*H_{\mb{Q}}^e)\right)>0,
 \end{aligned}
\]
 where $A=\mc{A}_{\mb{Q}}$, and $\mr{topdeg}(\cdot)$ is the topological degree.
 Applying \cite[Theorem 2.5]{song2023high} to $v^*\pi_4^*\ov{\mc{A}}^k\wt{\pi}^*\ov{H}^e$ and $u^*\pi_2^*\ov{\mc{A}}^k \wt{\pi}^*\ov{H}^e$, it follows that
	\[v^*\pi_4^*\ov{\mc{A}}^k \wt{\pi}^*\ov{H}^e\leq_{\mathrm n} (d+2-k)^k\frac{\big(u^*\pi_2^*\ov{\mc{A}}^{d+1-k}\cdot v^*\pi_4^*\ov{\mc{A}}^k\cdot \wt{\pi}^*\ov{H}^e\big)}{\big(u^*\pi_2^*\ov{\mc{A}}^{d+1}\cdot \wt{\pi}^*\ov{H}^e\big)}\cdot u^*\pi_2^*\ov{\mc{A}}^k\wt{\pi}^*\ov{H}^e.\]
	Intersecting both sides of the inequality above with $u^*\pi_1^*\ov{\mc{A}}^{d+1-k}$, one gets
	\[\ardeg_{k,\ov{\mc{A}}}(F\circ G)\leq (d+2-k)^k\frac{\big(u^*\pi_2^*\ov{\mc{A}}^{d+1-k}\cdot v^*\pi_4^*\ov{\mc{A}}^k\cdot \wt{\pi}^*\ov{H}^e\big)}{\big(u^*\pi_2^*\ov{\mc{A}}^{d+1}\cdot \wt{\pi}^*\ov{H}^e\big)} \ardeg_{k,\ov{\mc{A}}}(G).\]
	Note that $\pi_2\circ u=\pi_3\circ v$. We have
 \[\begin{aligned}
     \frac{\big(u^*\pi_2^*\ov{\mc{A}}^{d+1-k}\cdot v^*\pi_4^*\ov{\mc{A}}^k\cdot \wt{\pi}^*\ov{H}^e\big)}{\big(u^*\pi_2^*\ov{\mc{A}}^{d+1}\cdot \wt{\pi}^*\ov{H}^e\big)}&= \frac{\big(v^*\pi_3^*\ov{\mc{A}}^{d+1-k}\cdot v^*\pi_4^*\ov{\mc{A}}^k\cdot v^*\pi_3^*\pi^*\ov{H}^e\big)}{\big(v^*\pi_3^*\ov{\mc{A}}^{d+1}\cdot  v^*\pi_3^*\pi^*\ov{H}^e\big)}\\
     &=\frac{\mr{topdeg}(v)\cdot\big(\pi_3^*\ov{\mc{A}}^{d+1-k}\cdot \pi_4^*\ov{\mc{A}}^k\cdot \pi_3^*\pi^*\ov{H}^e\big)}{\mr{topdeg}(v)\cdot\big(\pi_3^*\ov{\mc{A}}^{d+1}\cdot \pi_3^*\pi^*\ov{H}^e\big)}\\
     &=\frac{\ardeg_{k,\ov{\mc{A}}}(F)}{\big(\ov{\mc{A}}^{d+1}\cdot \pi^*\ov{H}^e\big)}.
 \end{aligned}\]
	Thus
	\[\ardeg_{k,\ov{\mc{A}}}(F\circ G)\leq C
	\cdot \ardeg_{k,\ov{\mc{A}}}(F)\cdot\ardeg_{k,\ov{\mc{A}}}(G),\]
	where $C=\frac{(d+2-k)^k}{\big(\ov{\mc{A}}^{d+1}\cdot \pi^*\ov{H}^e\big)}>0$ is a constant independent of $F$ and $G$.
\end{proof}

\begin{defiprop}\label{coro_dyn_arith_deg}
For a dominant rational self-map $f:X\dra X$, the \emph{$k$-th arithmetic dynamical degree}
	\[\widehat{\lambda}_k(f):=\lim\limits_{n\ra \infty}\ardeg_{k,\ov{\mc{A}}}(F^n)^{1/n}.\]
 exists. In other words, the limit defining $\widehat{\lambda}_k(f)$ converges.
\end{defiprop}
\begin{proof}
The following Fekete's lemma is well-known.

\textbf{Fekete's Lemma:} Let $\{a_n\}_{n \geq 0}$ be a submultiplicative sequence of positive real numbers, meaning that for all $m, n \geq 0$,
$$
a_{n+m} \leq a_n a_m.
$$
Then the limit $\lim\limits_{n \rightarrow \infty} a_n^{\frac{1}{n}}$ exists.

We set $a_n:=C\cdot\ardeg_{k,\ov{\mc{A}}}(F^n)$, where $C$ is the positive constant from Theorem \ref{theo_deg_submulti}. This makes $\{a_n\}_{n \geq 0}$ a submultiplicative sequence of positive real numbers. The result then follows directly from Fekete's lemma.
\end{proof}
\begin{rema}
\begin{enumerate}
    \item When $K$ is a number field, we have $e=0$, then the arithmetic dynamical degree $\widehat{\lambda}_k(f)$ coincides with the arithmetic degree $\alpha_k(f)$ defined in \cite{DGH2022high} and \cite{song2023high}.
    \item  While the degree $\ardeg_{k,\ov{\mc{A}}}(F)$ a priori depends on the model triple $(\mc{X},F,\ov{\mc{A}})$,
  the arithmetic dynamical degree $\widehat \lambda_k(f)$ is independent of these choices. We will see this fact later in Theorem \ref{rel_deg_fgf}.
\end{enumerate}  
\end{rema}

In the remainder of this subsection, we establish the relative degree formula. Our strategy is to control the growth rate of $\ardeg_{k,\ov{\mc{A}}}(F^n)$ using a generalized form of the Yuan's inequality. The key ingredient comes from the following lemma, which is an extension of \cite[Lemma 3.1]{song2023high}. The proof of this lemma will be placed at the end of this subsection.

\begin{lemm}\label{lem_Siu}
Let $\mathcal{X}\ra \Spe\mb{Z}$ be an arithmetic variety with relative dimension $d$, and $r, k$ be integers with $r\geq 0, k>0, r+k\leq d$.  Let $\ov{\mc{L}},\ov{\mc{M}},\ov{\mc{N}}$ be nef hermitian line bundles with generic fibers $L=\mc{L}_{\mb{Q}}, M=\mc{M}_{\mb{Q}}$ and $N=\mc{N}_{\mb{Q}}$. Suppose $(L^k\cdot M^{d-r-k}\cdot N^r)>0$ and $(M^{d-r}\cdot N^r)>0$. Then there exists a positive constant $t_0>0$ only depends on $d, k, \ov{\mc{L}},\ov{\mc{M}}$ and $\ov{\mc{N}}$,  such that for any arithmetic variety $\mc{Y}\ra \Spe \mb{Z}$ with $\dim\mc{Y}=\dim\mc{X}$ and any dominant generically finite morphism $g:\mc{Y}\ra \mc{X}$, we have
	\[ g^*\ov{\mc{L}}^k\cdot g^*\ov{\mc{N}}^r\leq_{\mathrm n}
	(d+2-r-k)^k\frac{\big(L^k\cdot M^{d-r-k}\cdot N^r\big)}{\big(M^{d-r}\cdot N^r\big)}\cdot g^*\ov{\mc{M}}(t_0)^{k}\cdot g^*\ov{\mc{N}}^r.\]
Here, $\ov{\mc{M}}(t_0)=(\mc{M},e^{-t_0}\|\cdot\|)$ denotes the hermitian line bundle $\ov{\mc{M}}$ twisted by $t_0$.
\end{lemm}

\begin{theo}\label{rel_deg_fgf}
With the notations introduced at the beginning of this section, for a dominant rational self-map $f:X \dra X$, the $k$-th arithmetic dynamical degree can be expressed by
\[\widehat{\lambda}_k(f)=\max\{\lambda_k(f),\lambda_{k-1}(f)\},\]
where the dynamical degrees $\lambda_k(f)$ and $\lambda_{k-1}(f)$ are defined in the introduction.
\end{theo}

\begin{proof}
In the proof, we make use of the notation for the relative dynamical degree. Define $$\reldeg_{k,A}(F_{\mb{Q}},\mc{X}_{\mb{Q}}/B_{\mb{Q}}):=\big(\pi_2^*A^k\cdot \pi_1^*A^{d-k}\cdot\pi_1^*\pi^*H_{\mb{Q}}^e\big),$$
where $\pi: \mc{X}_{\mb{Q}}\ra B_{\mb{Q}}$ is the natural map between generic fibers.
The $k$-th \emph{relative dynamical degree} is then given by
\[\lambda_k(F_{\mb{Q}},\mc{X}_{\mb{Q}}/B_{\mb{Q}}):=\lim_{n\ra\infty}\reldeg_{k,A}(F_{\mb{Q}}^n,\mc{X}_{\mb{Q}}/B_{\mb{Q}})^\frac{1}{n},\]
where the existence of the limit is proved in \cite[Theorem 6.0.2]{dang_degrees_2020}.

We begin by proving that $\widehat{\lambda}_k(f)\geq \max\{\lambda_k(f),\lambda_{k-1}(f)\}$. Fix a real number $t>0$.  Applying \cite[Corollary 2.4]{song2023high} to $\pi_1^*\ov{\mO}_{\mc{X}}(t)$ and $\pi_1^*\ov{\mc{A}}$, we obtain
\[\pi_1^*\ov{\mO}_{\mc{X}}(t)\leq_{\mathrm n} (d+e+1)\frac{\big(\pi_1^*\ov{\mO}_{\mc{X}}(t)\cdot \pi_1^*\ov{\mc{A}}^{d+e}\big)}{\big(\pi_1^*\ov{\mc{A}}^{d+e+1}\big)}\cdot\pi_1^*\ov{\mc{A}}=(d+e+1)\frac{t\cdot\big(A^{d+e}\big)}{\big(\ov{\mc{A}}^{d+e+1}\big)}\cdot\pi_1^*\ov{\mc{A}}.\]
This inequality is, in fact, an equivalent form of Yuan’s inequality \cite[Theorem 2.2]{Yuan_2008}, which is explained and generalized in \cite[Corollary 2.4]{song2023high}.

We intersect both sides of the inequality with $\pi_2^*\ov{\mc{A}}^k\cdot \pi_1^*\ov{\mc{A}}^{d-k}\cdot \pi_1^*\pi^*\ov{H}^e$, it follows that
\[
	\begin{aligned}
		\ardeg_{k,\ov{\mc{A}}}(F)&=\big(\pi_2^*\ov{\mc{A}}^k\cdot \pi_1^*\ov{\mc{A}}^{d+1-k}\cdot \pi_1^*\pi^*\ov{H}^e\big)\\
		&\geq \big(\pi_2^*\ov{\mc{A}}^k\cdot \pi_1^*\ov{\mc{A}}^{d-k}\cdot \pi_1^*\pi^*\ov{H}^e\cdot \pi_1^*\ov{\mO}_{\mc{X}}(t)\big)\cdot\frac{ \big(\ov{\mc{A}}^{d+e+1}\big)}{(d+e+1)\cdot t\cdot\big(A^{d+e}\big)}\\
		&=C\cdot\big(\pi_2^*A^k\cdot \pi_1^*A^{d-k}\cdot\pi_1^*\pi^*H_{\mb{Q}}^e\big)\\
		&=C\cdot \reldeg_{k,A}(F_{\mb{Q}},\mc{X}_{\mb{Q}}/B_{\mb{Q}}),
	\end{aligned}
\]
where the constant $C:=\frac{ \big(\ov{\mc{A}}^{d+e+1}\big)}{(d+e+1)(A^{d+e})}$.

Similarly, we have
\[\pi_2^*\ov{\mO}_{\mc{X}}(t)\leq_{\mathrm n} (d+e+1)\frac{\big(\pi_2^*\ov{\mO}_{\mc{X}}(t)\cdot \pi_2^*\ov{\mc{A}}^{d+e}\big)}{\big(\pi_2^*\ov{\mc{A}}^{d+e+1}\big)}\cdot\pi_2^*\ov{\mc{A}}=(d+e+1)\frac{t\cdot\big(A^{d+e}\big)}{\big(\ov{\mc{A}}^{d+e+1}\big)}\cdot\pi_2^*\ov{\mc{A}}\]
by Yuan's inequality. Intersecting both sides of this inequality with $\pi_2^*\ov{\mc{A}}^{k-1}\cdot \pi_1^*\ov{\mc{A}}^{d+1-k}\cdot\pi_1^*\pi^*\ov{H}^e$, we obtain
\[
	\begin{aligned}
		\ardeg_{k,\ov{\mc{A}}}(F)&=\big(\pi_2^*\ov{\mc{A}}^k\cdot \pi_1^*\ov{\mc{A}}^{d+1-k}\cdot \pi_1^*\pi^*\ov{H}^e\big)\\
		&\geq \big(\pi_2^*\ov{\mc{A}}^{k-1}\cdot \pi_1^*\ov{\mc{A}}^{d+1-k}\cdot \pi_1^*\pi^*\ov{H}^e\cdot \pi_2^*\ov{\mO}_{\mc{X}}(t)\big)\cdot\frac{ \big(\ov{\mc{A}}^{d+e+1}\big)}{(d+e+1)\cdot t\cdot\big(A^{d+e}\big)}\\
		&=C\cdot\big(\pi_2^*A^{k-1}\cdot \pi_1^*A^{d+1-k}\cdot \pi_1^*\pi^*H_{\mb{Q}}^e\big)\\
		&=C\cdot \reldeg_{k-1,A}(F_{\mb{Q}},\mc{X}_{\mb{Q}}/B_{\mb{Q}}).
	\end{aligned}
\]
The two inequalities imply $$\widehat{\lambda}_k(f)\geq \max\{\lambda_k(F_{\mb{Q}},\mc{X}_{\mb{Q}}/B_{\mb{Q}}),\lambda_{k-1}(F_{\mb{Q}},\mc{X}_{\mb{Q}}/B_{\mb{Q}})\}.$$

Since $X$ is the generic fiber of $\mc{X}_{\mb{Q}}\ra B_{\mb{Q}}$, we have
\[\lambda_k(f)=\lambda_k(F_{\mb{Q}},\mc{X}_{\mb{Q}}/B_{\mb{Q}}), \text{ and } \lambda_{k-1}(f)=\lambda_{k-1}(F_{\mb{Q}},\mc{X}_{\mb{Q}}/B_{\mb{Q}}).\]
Thus we get \[\widehat{\lambda}_k(f)\geq \max\{\lambda_k(f),\lambda_{k-1}(f)\}.\]

Now we prove the converse inequality. We use the diagram \eqref{bigdiagram} again. According to Lemma \ref{lem_Siu} for the case of $r=0$, there exists a constant $t_0>0$ such that
\[ v^*\pi_4^*\ov{\mc{A}}^k\cdot \wt{\pi}^*\ov{H}^e\leq_{\mathrm n}
(d+2-k)^k\frac{\big(\pi_4^*A^k\cdot \pi_3^*A^{d-k}\cdot \pi_3^*\pi^*H_{\mb{Q}}^e\big)}{\big(A^{d}\cdot \pi^*H_{\mb{Q}}^e\big)}\cdot v^*\pi_3^*\ov{\mc{A}}(t_0)^{k}\cdot\wt{\pi}^*\ov{H}^e.\]
By intersecting the inequality above with $u^*\pi_1^*\ov{\mc{A}}^{d+1-k}$, it follows that
\begin{equation}\label{ineq_reldeg}
\begin{aligned}
	&\ardeg_{k,\ov{\mc{A}}}(F\circ G)\\
 =& \big(u^*\pi_1^*\ov{\mc{A}}^{d+1-k}\cdot v^*\pi_4^*\ov{\mc{A}}^k\cdot \wt{\pi}^*\ov{H}^e\big)\\
	\leq &(d+2-k)^k\frac{\big(\pi_4^*A^k\cdot \pi_3^*A^{d-k}\cdot \pi_3^*\pi^*H_{\mb{Q}}^e\big)}{\big(A^{d}\cdot \pi^*H_{\mb{Q}}^e\big)}\cdot \big(u^*\pi_1^*\ov{\mc{A}}^{d+1-k}\cdot v^*\pi_3^*\ov{\mc{A}}(t_0)^{k}\cdot \wt{\pi}^*\ov{H}^e\big)\\
	=&C_1\reldeg_{k,A}(F_{\mb{Q}},\mc{X}_{\mb{Q}}/B_{\mb{Q}}) \big(\big(u^*\pi_1^*\ov{\mc{A}}^{d+1-k}\cdot v^*\pi_3^*\ov{\mc{A}}^k\cdot \wt{\pi}^*\ov{H}^e\big)\\
    &\kern 12em +kt_0(u^*\pi_1^*A^{d+1-k}\cdot v^*\pi_3^*A^{k-1}\cdot \wt{\pi}^*H^e)\big)\\
	=&C_1\reldeg_{k,A}(F_{\mb{Q}},\mc{X}_{\mb{Q}}/B_{\mb{Q}}) \big(\big(u^*\pi_1^*\ov{\mc{A}}^{d+1-k}\cdot u^*\pi_2^*\ov{\mc{A}}^k\cdot \wt{\pi}^*\ov{H}^e\big)\\&\kern 12em+kt_0(u^*\pi_1^*A^{d+1-k}\cdot u^*\pi_2^*A^{k-1}\cdot \wt{\pi}^*H^e)\big)\\
	=& C_2\deg_{k,A}(f)  \big(\ardeg_{k,\ov{\mc{A}}}(G)+kt_0(H_{\mb{Q}}^e)\deg_{k,A}(g) \big),
\end{aligned}
\end{equation}
where $C_2=C_1(H_{\mb{Q}}^e)=(H_{\mb{Q}}^e)\cdot\frac{(d+2-k)^k}{(A^{d}\cdot \pi^*H_{\mb{Q}}^e)}$ is a constant independent of the choice of $F$ and $G$. 

According to the submultiplicity formula of dynamical degrees in \cite[Theorem 1]{dang_degrees_2020}, we have
\[\deg_{k-1,A}(f\circ g)\leq C_3 \deg_{k-1,A}(f)\cdot \deg_{k-1,A}(g),\]
where the constant $C_3$ is independent of the choice of $f$ and $g$.

We construct the matrices
\begin{equation}
U(F):=\begin{bmatrix} \widehat{\deg}_{k,\ov{\mc{A}}}(F)  \\ \deg_{k-1,A}(f)  \end{bmatrix},\quad M(f):=\begin{bmatrix} C_2\deg_{k,A}(f) & C_2k t_0 (H_{\mb{Q}}^e)\deg_{k,A}(f) \\ 0 & C_3\deg_{k-1,A}(f)\end{bmatrix}.
\end{equation}
For clarity, we adopt the convention that for any vectors $v=(v_1,v_2)^{\top}, w=(w_1,w_2)^{\top}\in \mb{R}^2$, the relation $v\leq w$ means $v_1\leq w_1$ and $v_2\leq w_2$.

Next, by setting $G=F^{n-1}$ in \eqref{ineq_reldeg}, we obtain
\[U(F^{n})\leq M(f)\cdot U(F^{n-1}).\]
Iterating this inequality then yields
\[U(F^{n})\leq M(f)^{n-1}\cdot U(F).\]
Consider the norm on $\mb{R}^2$ defined by
\[\|(v_1,v_2)^{\top}\|_{\infty}:=\max\{|v_1|,|v_2|\},\]
and for $M=(a_{ij})_{1\leq i,j\leq 2}\in M_2(\mb{R})$, define 
\[\|M\|_{\mr{max}}:=\max\{|a_{ij}|,1\leq i,j\leq 2\}.\]
With this notation, we observe that
\[\widehat{\deg}_{k,\ov{\mc{A}}}(F^n)\leq\|U(F^n)\|_{\infty}\leq \| M(f)^{n-1}\cdot U(F)\|_{\infty}\leq 2\|M(f)^{n-1}\|_{\mr{max}}\|U(F)\|_{\infty}.\]
Replacing $F$ by $F^r$ in the preceding inequality and taking the $nr$-th root yields
 \[\widehat{\deg}_{k,\ov{\mc{A}}}(F^{nr})^{1/nr}\leq 2^{1/nr}\cdot\|U(F^r)\|_{\infty}^{1/nr}\|M(f^r)^{n-1}\|_{\mr{max}}^{1/nr}.\]
Letting $n\ra \infty$ and noting that the eigenvalues of $M(f^r)$ are $C_2\deg_{k,A}(f^r)$ and $C_3\deg_{k-1,A}(f^r)$, an application of the spectral radius formula shows that
\[\lim_{n\ra\infty}\widehat{\deg}_{k,\ov{\mc{A}}}(F^{nr})^{1/nr}=\widehat{\lambda}_k(f)\leq \max\{C_2\deg_{k,A}(f^r),C_3\deg_{k-1,A}(f^r)\}^{1/r}.\]
Finally, taking the limit as $r\ra\infty$ yields the desired inequality
\[\widehat{\lambda}_k(f)\leq \max\{\lambda_k(f),\lambda_{k-1}(f)\}.\]
Hence we have the assertion.
\end{proof}

The following corollary is an immediate consequence of Theorem \ref{rel_deg_fgf}.

\begin{coro}
With all notations in Definition \ref{def_ar_dyn_deg}, the arithmetic dynamical degree $\widehat{\lambda}_k(f)$ is independent of the choice of models $\ov{B}=(B;\ov{H}), \mc{X}$ or the ample hermitian line bundle $\ov{\mc{A}}$ on $\mc{X}$.
\end{coro}

\textit{Proof of Lemma \ref{lem_Siu} } By applying \cite[Theorem 2.5]{song2023high} to $g^*\ov{\mc{L}}^k\cdot g^*\ov{\mc{N}}^r$ and $g^*\ov{\mc{M}}(t)^{k}\cdot g^*\ov{\mc{N}}^r$, we have
		\begin{align*}\allowdisplaybreaks
			&\kern 1em g^*\ov{\mc{L}}^k \cdot g^*\ov{\mc{N}}^r\\
            &\leq_{\mathrm n}(d+2-r-k)^k\frac{\big(g^*\ov{\mc{L}}^k\cdot g^*\ov{\mc{M}}(t)^{d+1-r-k}\cdot g^*\ov{\mc{N}}^r\big)}{\big(g^*\ov{\mc{M}}(t)^{d+1-r}\cdot g^*\ov{\mc{N}}^r\big)}\cdot g^*\ov{\mc{M}}(t)^{k}\cdot g^*\ov{\mc{N}}^r\\
			&=
			(d+2-r-k)^k\frac{\big(\ov{\mc{L}}^k\cdot \ov{\mc{M}}(t)^{d+1-r-k}\cdot \ov{\mc{N}}^r\big)}{\big(\ov{\mc{M}}(t)^{d+1-r}\cdot \ov{\mc{N}}^r\big)}\cdot g^*\ov{\mc{M}}(t)^{k}\cdot g^*\ov{\mc{N}}^r\\
			&=(d+2-r-k)^k\frac{\big(\ov{\mc{L}}^k\cdot \ov{\mc{M}}^{d+1-r-k}\cdot \ov{\mc{N}}^r\big)+(d+1-r-k)t(L^k\cdot M^{d-r-k}\cdot N^r)}{\big(\ov{\mc{M}}^{d+1-r}\cdot \ov{\mc{N}}^r\big)+(d+1-r)t(M^{d-r}\cdot N^r)}\cdot g^*\ov{\mc{M}}(t)^{k}\cdot g^*\ov{\mc{N}}^r.
		\end{align*}
We take $t_0\in \mb{R}$ sufficiently large such that
	\[
	    \frac{\big(\ov{\mc{L}}^k\cdot \ov{\mc{M}}^{d+1-r-k}\cdot \ov{\mc{N}}^r\big)+(d+1-r-k)t_0(L^k\cdot M^{d-r-k}\cdot N^r)}{\big(\ov{\mc{M}}^{d+1-r}\cdot \ov{\mc{N}}^r\big)+(d+1-r)t_0(M^{d-r}\cdot N^r)}
     <\frac{\big(L^k\cdot M^{d-r-k}\cdot N^r\big)}{\big(M^{d-r}\cdot N^r\big)}.
\]
	Thus we obtain
	\[\begin{aligned}
	&\frac{\big(\ov{\mc{L}}^k\cdot \ov{\mc{M}}^{d+1-r-k}\cdot \ov{\mc{N}}^r\big)+(d+1-r-k)t(L^k\cdot M^{d-r-k}\cdot N^r)}{\big(\ov{\mc{M}}^{d+1-r}\cdot \ov{\mc{N}}^r\big)+(d+1-r)t(M^{d-r}\cdot N^r)}\cdot g^*\ov{\mc{M}}(t_0)^{k}\cdot g^*\ov{\mc{N}}^r \\
     \leq_{\mathrm n} &\frac{\big(L^k\cdot M^{d-r-k}\cdot N^r\big)}{\big(M^{d-r}\cdot N^r\big)}\cdot g^*\ov{\mc{M}}(t_0)^{k}\cdot g^*\ov{\mc{N}}^r.
	\end{aligned}
\]
	This completes the proof.\qed

\subsection{Arithmetic degrees and the Kawaguchi--Silverman conjecture}
As in the number field case, one is able to define arithmetic degrees for $x\in X(\ov{K})$, which measures the arithmetic complexity of the $f$-orbit of this point. 

\begin{defi}\label{arith_deg_fgf}
Let $f: X\dra X$ be a dominant rational self-map of $X$. Fix a Moriwaki height $h^{\ov{B}}_{(\mc{X},\ov{\mc{A}})}$ such that $\ov{\mc{A}}$ is an ample hermitian line bundle on $\mc{X}$. Let $h^+_A=\max\left\{h^{\ov{B}}_{(\mc{X},\ov{\mc{A}})},1\right\}$. Then the \emph{arithmetic degree} is defined as
\[\alpha(f,x)=\lim_{n\ra\infty}h^+_{(\mc{X},\ov{\mc{A}})}(f^n(x))^{1/n}\]
for all $x\in X_f(\ov{K})$ if the limit exists.
In addition, we denote $\ov{\alpha}(f,x)=\limsup\limits_{n\ra\infty}h^+_{(\mc{X},\ov{\mc{A}})}(f^n(x))^{1/n}$ and $\underline{\alpha}(f,x)=\liminf\limits_{n\ra\infty}h^+_{(\mc{X},\ov{\mc{A}})}(f^n(x))^{1/n}$.
\end{defi}

\begin{rema}\label{rema_ar_deg}
\begin{enumerate}
     \item According to Theorem \ref{theo_height_comparable}, $\ov{\alpha}(f,x)$ and $ \underline{\alpha}(f,x)$ are independent of the choice of models $\ov{B}=(B;\ov{H})$ and $(\mc{X},\ov{\mc{A}})$.
     
    \item When $K$ is a number field, the arithmetic degree $\alpha(f, x)$ coincides with the classical arithmetic degree. For further details, see \cite{matsuzawa2023recent}. However, if $K$ is a function field over $\mb{Q}$ of one variable, our arithmetic degree is different from the arithmetic degree defined by the geometric height function, see Example \ref{exam_func}.
\end{enumerate}
\end{rema}

We can establish the fundamental inequality between arithmetic degrees and dynamical degrees.
\begin{theo}\label{upper_bound_fgf}
Let $K$ be a finitely generated field over $\mb{Q}$ and $X\ra \Spe K$ be a normal projective variety. For a dominant rational self-map $f:X\dra X$, we have
\[\ov{\alpha}(f,x)\leq \lambda_1(f)\]
for all $x\in X_f(\ov{K})$.
\end{theo}
\begin{proof}If $x\in X_f(\ov{K})$ is not defined over $K$, we consider the base change $X_{L}=X\times_{\Spe K}\Spe L$, where $L=K(x)$. Then the rational map $f:X\dra X$ lifts to a rational map $f_L: X_L \dra X_L$ and  $x_L$ is a rational point of $X_L$ defined over $L$. By \cite[Proposition 3.3.1]{moriwaki2000arithmetic}, it follows that $\ov{\alpha}(f,x)=\ov{\alpha}(f_L,x_L)$. Thus, by replacing $(X,f,x)$ with $(X_L,f_L,x_L)$, we may assume that $[K(x): K]=1$. 

Take a big and nef polarization $\ov{B}=(B,\ov{H})$ and a $B$-model $(\mc{X},\ov{\mc{A}})$ of $(X,A)$, where $A$ is an ample line bundle of $X$. Suppose $\dim X=d$ and $\mr{tr.deg}(K/\mb{Q})=e$, then $\mc{X}$ and $B$ are arithmetic varieties of dimension $d+e+1$ and $e+1$ respectively. According to Theorem \ref{theo_height_comparable}, the arithmetic degree $\ov{\alpha}(f, x)$ is independent of the choice of the ample hermitian line bundle $\ov{\mc{A}}$. Then we can choose $\ov{\mc{A}}$ arbitrarily. 

Denote by $\Delta_x$ the Zariski closure of the image of $x: \Spe(\ov{K})\ra X\ra \mc{X}$ in $\mc{X}$. Then $\Delta_x$ is a horizontal subvariety of dimension $e+1$ of the arithmetic variety $\mc{X}\ra \Spe\mb{Z}$. Denote by $V\subset \mc{X}_{\mb{Q}}$ the generic fiber of $\Delta_x$, which is also the Zariski closure of the image of $x:\Spe(\ov{K})\ra X\ra \mc{X}_{\mb{Q}}$ in $\mc{X}_{\mb{Q}}$. Let $I$ be the ideal sheaf of $V$. Then we replace $\ov{\mc{A}}$ by $\ov{\mc{A}}^{\otimes r}$ for $r$ sufficient large such that $\mc{A}_{\mb{Q}}\otimes I$ is globally generated.

We take global sections $s_1,s_2,\dots,s_{d}\in H^0(\mc{X}_{\mb{Q}},\mc{A}_{\mb{Q}}\otimes I)\subset H^0(\mc{X}_{\mb{Q}},\mc{A}_{\mb{Q}})$ such that $\dim(\di(s_1)\cap \dots \cap \di(s_{d}))=e$.
After replacing $s_i$ by $m\cdot s_i$ for some integer $m$ sufficiently large, we may assume that all sections $s_1,\dots,s_{d}$ lift to global sections $\wt{s}_1,\dots,\wt{s}_{d}\in H^0(\mc{X},\mc{A}).$
The intersection cycle $\di(\wt{s}_1)\cap \dots \cap \di(\wt{s}_{d})$ then contains $\Delta_x$ as an irreducible component. Furthermore, by replacing $\ov{\mc{A}}$ with $\ov{\mc{A}}(t)$ for a sufficiently large $t$, we can ensure that $\wt{s}_1,\dots,\wt{s}_{d}$ are small sections of $\ov{\mc{A}}$. 

Since $\widehat{\lambda}_1(f)=\lim\limits_{n\ra\infty}\widehat{\deg}_{1,\ov{\mc{A}}}(F^n)^{1/n}$, for any $\varepsilon>0$, there exists a constant $C_{\varepsilon}>0$, depending only on $\varepsilon$, such that
\[\widehat{\deg}_{1,\ov{\mc{A}}}(F^n)\leq C_{\varepsilon}(\widehat{\lambda}_1(f)+\varepsilon)^n.\]
Now we take a resolution of indeterminacy of $F^n$
 \[
\begin{tikzcd}
 & \Gamma_{F^n} \arrow{dl}[swap]{\pi_1} \arrow{dr}{\pi_2}\\
\mc{X} \arrow[dashed]{rr}{F^n} && \mc{X} .
\end{tikzcd}
\]
Note that $\pi_1^{\#}\Delta_x$ is an irreducible component of $\di(\pi_1^*\wt{s}_1)\cap \dots \cap \di(\pi_1^*\wt{s}_{d})$, where $\pi_1^{\#}\Delta_x$ is the strict transform of $\Delta_x$ in $\Gamma_{F^n}$. According to \cite[Lemma 4.5]{song2023high}, it follows that
\begin{equation}\label{ineq_fundamental}
    \begin{aligned}
     \big(\ov{\mc{A}}\cdot \pi^*\ov{H}^e\cdot \Delta_{f^n(x)}\big)&=\big(\pi_2^*\ov{\mc{A}}\cdot\pi_2^*\pi^*\ov{H}^e\cdot \pi_1^{\#}\Delta_x\big)\\
     &\leq \big(\pi_2^*\ov{\mc{A}}\cdot\pi_2^*\pi^*\ov{H}^e\cdot \pi_1^*\ov{\mc{A}}^d\big)\\
     &=\widehat{\deg}_{1}(F^n)\\
     &\leq C_{\varepsilon}(\widehat{\lambda}_1(f)+\varepsilon)^n.
\end{aligned}
\end{equation}
Thus 
\[\begin{aligned}
    \ov{\alpha}(f,x)&=\limsup\limits_{n\ra\infty}h^{\ov{B}}_{(\mc{X},\ov{\mc{A}})}(f^n(x))\\
    &=\limsup\limits_{n\ra\infty}\big(\ov{\mc{A}}\cdot \pi^*\ov{H}^e\cdot \Delta_{f^n(x)}\big)^{1/n}\\
    &\leq \limsup_{n\ra\infty} C_{\varepsilon}^{1/n}(\widehat{\lambda}_1(f)+\varepsilon)\\
    &= \widehat{\lambda}_1(f)+\varepsilon.
\end{aligned}
\]

Let $\varepsilon\ra 0$ and applying Theorem \ref{rel_deg_fgf}, we obtain
\[\ov{\alpha}(f,x)\leq \widehat{\lambda}_1(f)=\max\{\lambda_1(f),1\}=\lambda_1(f).\]
\end{proof}

Motivated by the Kawaguchi–-Silverman conjecture over number fields, we propose the following conjecture.

\begin{conj}[The Kawaguchi--Silverman conjecture over a finitely generated field $K$ over $\mb{Q}$]
Let $X$ be a normal projective variety over a finitely generated field $K$ and $f: X\dra X$ be a dominant rational self-map. Then
\begin{itemize}
    \item the limit $\alpha(f,x)$ in Definition \ref{arith_deg_fgf} exists for every $x\in X_f(\ov{K})$;
    \item for $x\in X_f(\ov{K})$, if the $f$-orbit of $x$ is Zariski dense in $X$, then $\alpha(f,x)=\lambda_1(f)$.
\end{itemize}
\end{conj}

\begin{theo}\label{theo_KSC_polarized}
    Suppose that there exists a big $\R$-Cartier $D$ such that $f^*D\equiv \lambda_1(f) D$, where $\equiv$ stands for the numerical equivalence. Then the Kawaguchi--Silverman conjecture holds.
\end{theo}
\begin{proof}
    When $\lambda_1(f)=1$, we have
    \[1\leq \underline{\alpha}(f,x)\leq \ov{\alpha}(f,x)\leq \lambda_1(f)=1.\]
    The consequence follows immediately. Now we assume $\lambda_1(f)>1$.
    
    Since $D$ is big, there exists an ample divisor $A$ and an effective divisor $E$ such that $D\sim A+E$. 
    We fix height functions on $A$ and $E$ such that
    $$h_{\ovl A}\geq 1\text{ on }X(\ovl K), h_{\ovl E}\geq 0\text{ on }(X\setminus \mathrm{Supp}(E))(\ovl K).$$
    Let $x\in X_f(\ovl K)$ be an algebraic point with Zariski dense orbit. It suffices to show that $\underline \alpha(f,x)\geq \lambda_1(f)$. 
    By Proposition \ref{prop_pullback} and Proposition \ref{prop_num_triv}, there exists a constant $C_1>0$ such that 
    $$h_{\ovl D}(f(x))\geq \lambda_1(f) h_{\ovl D}(x)-C_1\sqrt{h_{\ovl A}(x)},$$
    which gives that
    \begin{equation}\label{ineq_height_D}
        h_{\ovl D}(f^n(x))\geq \lambda_1(f)^n h_{\ovl D}(x)-C_1\sum_{i=1}^n\lambda_1(f)^{n-i}\sqrt{h_{\ovl A}(f^{i-1}(x))}.
    \end{equation}
    Notice that by the fundamental inequality (see \cite[Theorem 5.2.1]{ohnishi2022arakelov}, the proof is similar to \cite[Theorem 2.1, Theorem 3.1]{matsuzawa2017bounds}), for any $\epsilon>0$, there exists $C_\epsilon$ such that
    \begin{equation}\label{ineq_height_A}
        h_{\ovl A}(f^n(x))\leq C_\epsilon (\lambda_1(f)+\epsilon)^nh_{\ovl A}(x).
    \end{equation}
    We set that $\epsilon=(\lambda_1(f)^2-\lambda_1(f))/2$ and $C_2=C_1\sqrt{C_\epsilon}/\lambda_1(f)$. Then \eqref{ineq_height_D} and \eqref{ineq_height_A} gives that
    \begin{equation}\label{ineq_height_DA}
        \begin{aligned}
        h_{\ovl D}(f^n(x))&\geq \lambda_1(f)^n \left(h_{\ovl D}(x)-C_2\sum_{i=1}^n\Big(\frac{\sqrt{\lambda_1(f)+\epsilon}}{\lambda_1(f)}\Big)^{i-1}\sqrt{h_{\ovl A}(x)}\right)\\
        &\geq \lambda_1(f)^n\left(h_{\ovl D}(x)-C_2\Big(1-\frac{\sqrt{\lambda_1(f)+\epsilon}}{\lambda_1(f)}\Big)^{-1}\sqrt{h_{\ovl A}(x)}\right).
        \end{aligned}
    \end{equation}
    Since $\mO_f(x)$ is Zariski dense, by Proposition \ref{prop_unbound_big}, $\{h_{\ovl A}(f^n(x))\}_{n\in\N}$ is not bounded from above. Therefore we may assume that $\displaystyle h_{\ovl A}(x)-C_3^2>0$ and $x\in (X\setminus \mathrm{Supp}(E))(\ovl K)$, where $C_3=C_2\Big(1-\frac{\sqrt{\lambda_1(f)+\epsilon}}{\lambda_1(f)}\Big)^{-1}.$ Then $h_{\ovl D}(x)-C_3\sqrt{h_{\ovl A}(x)}\geq h_{\ovl A}(x)-C_3\sqrt{h_{\ovl A}(x)}>0$, which concludes the proof by taking $n$-th root of both sides of \eqref{ineq_height_DA}, and letting $n\rightarrow +\infty$.
\end{proof}

Notice that the proof actually gives the following result.
\begin{coro}\label{coro_canonicalheight}
    Let $f: X\rightarrow X$ be an endomorphism and $D$ be a nef $\R$-Cartier divisor such that $f^*D\equiv \lambda_1(f) D$. Then for each $x\in X(\ov K)$, the limit $$\widehat{h}_D(x):=\lim_{n\rightarrow\infty} \frac{h_{\ov D}(f^n(x))}{\lambda_1(f)^n}$$ converges. We call $\widehat h_D$ the \textit{canonical height} with respect to $D$. Moreover, $\widehat{h}_D(x)=h_{\ov D}(x)+O(\sqrt{h_{\ov A}(x)})$ where $A$ is an ample Cartier divisor and $h_{\ov A}\geq 1$. 
\end{coro}

\begin{rema}
    It would be worthwhile to check that the proofs for Kawaguchi--Silverman conjectures in the known cases, for example \cite{Kawaguchi2016KSC,KSC2018Surface,KSC2020Fano,KSC2024Hyperbolic}, can be generalized for finitely generated fields over $\Q$. In the thesis of Ohnishi \cite{ohnishi2022arakelov}, the cases of regular affine automorphisms and surface automorphisms are studied. Moreover, when $f$ is an endomorphism, Ohnishi proves the existence of $\alpha(f,x)$, which extends \cite[Theorem 3]{KS2016Jordan}. 
\end{rema}

\section{Arithmetic degrees over fields of characteristic zero}\label{Sec_char0}
In this section, let $\mathbf{k}$ be an algebraically closed field of characteristic $0$. 
For example, we can choose $\mathbf{k}$ to be $\mb{C},\mb{C}_p$ or $\ov{\C(t)}$.

Let $X$ be a normal projective variety over $\mathbf{k}$ and $f: X\dra X$ be a dominant rational self-map defined over $\mathbf{k}$. For a given point $x\in X(\mathbf{k})$, we can choose a finitely generated field $K$ over $\mb{Q}$ such that the triple $(X,f,x)$ can be defined over $K$ \cite[Proposition 8.9.I]{EGA4}. More precisely, let $K$ be the extension of $\mb{Q}$ that is generated by all coefficients of polynomials defining $(X,f)$ and all coordinates of $x\in X_f(\mathbf{k})$. Then there is a projective variety $X_K$ over $K$, a dominant rational self-map $f_K$ of $X_K$, and a point $x_K\in X(K)$, such that the base change of $(X_K, f_K, x_K)$ to $\mathbf{k}$ recovers $(X,f,x)$. We have the following diagram:
\[
\begin{tikzcd}
                                                                & X \arrow{dd} \arrow[dashed]{rr}{f} \arrow{rd} &                           & X \arrow{dd} \arrow{ld} \\
\Spe\mathbf{k} \arrow{dd} \arrow{rr}{\quad\quad\mr{id}} \arrow{ru}{x} &                                                 & \Spe\mathbf{k} \arrow{dd} &                         \\
                                                                & X_K \arrow{rd} \arrow[dashed]{rr}{\quad \quad f_K}        &                           & X_K \arrow{ld}          \\
\Spe K \arrow{ru}{x_K} \arrow{rr}{\mr{id}}                  &                                                 & \Spe K                    &                        
\end{tikzcd}
\]
By abuse of terminology, we say that $(X,f,x)$ is \emph{defined over} $K$, if it is obtained as the base change of $(X_K, f_K, x_K)$ to $\mathbf{k}$.

\subsection{Properties of the arithmetic degree over fields of characteristic 0}
We are going to define the arithmetic degree of $(X,f,x)$.
\begin{defi}\label{def_arithdeg_C}
For a normal projective variety $X$ over $\mathbf{k}$, a dominant rational self-map $f$, and a rational point $x\in X_f(\mathbf{k})$, we choose a finitely generated field $K$ over $\mb{Q}$ such that $(X,f,x)$ is defined over $K$. Denote the corresponding triple over $K$ by $(X_K,f_K,x_K)$, then we define the arithmetic degree of $x$ as
\[\ov{\alpha}(f,x):=\ov{\alpha}(f_K,x_K), \text{and }\underline{\alpha}(f,x):=\underline{\alpha}(f_K,x_K).\]
The right-hand sides are arithmetic degrees over the finitely generated field $K$, which is defined in Definition \ref{arith_deg_fgf}. If $\ov{\alpha}(f,x)=\underline{\alpha}(f,x)$, it can be simply denoted by $\alpha(f,x)$.
\end{defi}

\begin{prop}\label{prop_arith_deg_intr}
For a triple $(X,f,x)$, the arithmetic degrees $\ov{\alpha}(f,x)$ and $\underline{\alpha}(f,x)$ are independent of the choice of the base field $K$. In other words, the definition of arithmetic degree is intrinsic.
\end{prop}
\begin{proof}
Suppose the triple $(X,f,x)$ is defined on finitely generated field $K_1$ and $K_2$. 
Then it is defined on $K_0=K_1\cap K_2$. The triple $(X,f,x)$ comes from $(X_1,f_1,x_{K_1})$ and $(X_2,f_2,x_{K_2})$ by base change.

We choose models $B_1$ and $B_2$ of $K_1$ and $K_2$, respectively, and endow each $B_i$ an ample polarization $\ov H_i$ for $i=1,2$, as described in \S \ref{Sec_height_1}. We take $B_i$-model $\mc{X}_i$ of $X_i$  and let $\ov{\mc{A}}_i$ be an ample hermitian line bundle on $\mc{X}_i$ for $i=1,2$. Then by Theorem \ref{theo_height_comparable}, there exists constants $C_1,D_1,C_2,D_2>0$ such that
$$C_1 h_{(\scrX_1,\ovl{\mathcal A}_1)}^{\ovl B_1}(x_{K_1})-D_1\leq h_{(\scrX_2,\ovl {\mathcal A}_2)}^{\ovl B_2}(x_{K_2})\leq C_2 h_{(\scrX_1,\ovl {\mathcal A}_1)}^{\ovl B_1}(x_{K_1})+D_2.$$
Furthermore, the constants $C_1,D_1,C_2,D_2$ depend only on the data  $K_1, K_2, \mc{X}_1, \mc{X}_2, \ov{\mc{A}}_1, \ov{\mc{A}}_2, \ov{B}_1, \ov{B}_2$, and are thus independent of the choice of $x$. Replacing $x$ with $f^n(x)$, the inequality becomes
$$C_1 h_{(\scrX_1,\ovl{\mathcal A}_1)}^{\ovl B_1}\big(f_1^n(x_{K_1})\big)-D_1\leq h_{(\scrX_2,\ovl {\mathcal A}_2)}^{\ovl B_2}\big(f_2^n(x_{K_2})\big)\leq C_2 h_{(\scrX_1,\ovl {\mathcal A}_1)}^{\ovl B_1}\big(f_1^n(x_{K_1})\big)+D_2.$$
It follows that $\ov{\alpha}(f_1,x_{K_1})=\ov{\alpha}(f_2,x_{K_2})$ and $\underline{\alpha}(f_1,x_{K_1})=\underline{\alpha}(f_2,x_{K_2})$. In other words, the arithmetic degrees $\ov{\alpha}(f,x)$ and $\underline{\alpha}(f,x)$ are well-defined.
\end{proof}

\begin{theo}\label{theo_fundamental_ineq}
For all $\mathbf{k}$-point $x\in X_f(\mathbf{k})$, there is an upper bound
\[\overline{\alpha}(f,x)\leq \lambda_1(f).\]
\end{theo}
\begin{proof}
Suppose $(X,f,x)$ is defined over a finitely generated field $K$ over $\mb{Q}$ and comes from the triple $(X_K,f_K,x_K)$. Since dynamical degrees are invariant under base change, we have $\lambda_1(f)=\lambda_1(f_K)$. Thus by Theorem \ref{upper_bound_fgf}, the consequence follows. 
\end{proof}

Due to the theorem above, we propose a generalized version of the Kawaguchi--Silverman conjecture in our settings.
\begin{conj}
Let $X$ be a normal projective variety over a field $\mathbf{k}$ of characteristic $0$ and $f: X\dra X$ be a dominant rational self-map. Then
\begin{itemize}
    \item the limit defining $\alpha(f,x)$ exists for every $x\in X_f(\mathbf{k})$;
    \item for $x\in X_f(\mathbf{k})$, if the $f$-orbit $\mO_f(x)$ is Zariski dense in $X$, then $\alpha(f,x)=\lambda_1(f)$.
\end{itemize}    
\end{conj}

By spreading out to a finitely generated field, we will prove Proposition \ref{theo_KSC_C} for $(X,f,x)$ defined over $\mathbf{k}$.
\begin{prop}\label{theo_KSC_C}
Let $X$ be a normal projective variety over $\mathbf{k}$, and $f$ be an endomorphism of $X$. Suppose there exists a big $\R$-Cartier divisor $D$ such that $f^*D\equiv \lambda_1(f)D$. Then for $x\in X(\mathbf{k})_f$ whose $f$-orbit is Zariski dense, we have $\alpha(f,x)=\lambda_1(f)$.
\end{prop}
\begin{proof}
We choose a finitely generated field $K$ over $\mb{Q}$ such that the triple $(X,f,x)$ and the divisor $D$ are defined over $K$. Then by Theorem \ref{theo_KSC_polarized}, we obtain the assertion.
\end{proof}

In the original version of the Kawaguchi--Silverman conjecture \cite[Conjecture 6]{KS2016dyna}, it was conjectured that the set of values of arithmetic degrees is a finite set of algebraic integers, and this was proven for the case of endomorphism. In general, for a dominant rational self-map $f: X\dra X$, its arithmetic degrees may take infinitely many values, see \cite{LS20infinite}. If $f: X\ra X$ is an endomorphism,  Ohnishi generalized the result of Kawaguchi and Silverman in \cite{ohnishi2022arakelov}  to the case where $(X,f)$ is defined over adelic curves with the Northcott property, which are essentially finitely generated fields over $\Q$. We extend this result to $(X,f)$ over an arbitrary field of characteristic $0$. 
\begin{prop}\label{prop_finite_set}
    Let $X$ be a normal projective variety over $\mathbf{k}$ and $f: X\ra X$ be an endomorphism. Then
    \begin{itemize}
        \item the limit defining $\alpha(f,x)$ exists for every $x\in X(\mathbf{k})$;
        \item the arithmetic degree $\alpha(f,x)$ is an algebraic integer;
        \item the collection of arithmetic degrees
        \[\left\{\alpha(f,x)\mid x\in X(\mathbf{k})\right\}\subset \mb{R}\]
        is finite.
    \end{itemize}
\end{prop}
\begin{proof}
    Suppose $(X,f)$ is defined over a finitely generated field $K$ over $\mb{Q}$. Let $K'/K$ be a finitely generated extension. By \cite[Theorem 5.3.1]{ohnishi2022arakelov}, we have
    \begin{itemize}
        \item the limit defining $\alpha(f_{K'},x)$ exists for every $x\in X_{K'}(\ov{K'})$;
        \item the arithmetic degree $\alpha(f_{K'},x)$ is an algebraic integer;
        \item the collection of arithmetic degrees
        \[\left\{\alpha(f_{K'},x)\mid x\in X_{K'}(\ov{K'})\right\}\]
        is finite.
    \end{itemize}
For $x\in X(\mathbf{k})$, by definition, we have $\alpha(f,x)=\alpha(f_{K^{\prime}},x_{K^{\prime}})$ for an extension $K^{\prime}/K$ such that $x$ is defined over $K^{\prime}$. Consequently, the arithmetic degree $\alpha(f,x)$ exists and is an algebraic integer for any $x\in X(\mathbf{k})$. It remains to prove that the set 
\[\left\{\alpha(f,x)\mid x\in X(\mathbf{k})\right\}\]
is finite.

Suppose $x\in X(\mathbf{k})$. Note that the N\'eron-Severi group $\mr{NS}(X)_{\mb{Q}}$ is a finitely dimensional vector space. Let $D_1, D_2,\dots, D_m$ be a basis of $\mr{NS}(X)_{\mb{Q}}$. Suppose $D_1, D_2,\dots, D_m$ and $x$ are defined over a finitely generated field ${K'}$ such that $K\subset {K'}\subset \mathbf{k}$. Then there is a canonical isomorphism $\mr{NS}(X_{K'})_{\mb{Q}}\cra \mr{NS}(X)_{\mb{Q}}$, and it commutes with the action of $f_{K'}^*$ and $f^*$. There is a diagram:
\[
\begin{tikzcd}
\mr{NS}(X_{K'})_{\mb{Q}} \arrow{r}{\sim} \arrow{d}{f_{K'}^*} & \mr{NS}(X)_{\mb{Q}} \arrow{d}{f^*} \\
\mr{NS}(X_{K'})_{\mb{Q}} \arrow{r}{\sim}                             & \mr{NS}(X)_{\mb{Q}}   \end{tikzcd}
\]
 By the diagram above, eigenvalues of $f_{K'}^*:\mr{NS}(X_{K'})_{\mb{Q}}\ra \mr{NS}(X_{K'})_{\mb{Q}}$ are exactly eigenvalues of $f^*:\mr{NS}(X)_{\mb{Q}}\ra \mr{NS}(X)_{\mb{Q}}$. Moreover, the arithmetic degree $\alpha(f_{K'},x_{K'})$ equals to $1$ or $|\lambda|$, where $\lambda$ is an eigenvalue of $f_{K'}^*:\mr{NS}(X_{K'})_{\mb{Q}}\ra \mr{NS}(X_{K'})_{\mb{Q}}$. Thus the set 
\[\left\{\alpha(f,x)\mid x\in X(\mathbf{k})\right\}\subset \{1\}\cup \{|\lambda|\mid \lambda \text{ is an eigenvalue of } f^*:\mr{NS}(X)_{\mb{Q}}\ra \mr{NS}(X)_{\mb{Q}}\}\]
is finite.
\end{proof}

In \cite{Matsuzawa2018zero}, the arithmetic degree over a function field of characteristic $0$ is considered. To be specific, consider a projective variety $X$ defined over $K=\mathsf k(C)$, where $\mathsf k$ is a field of characteristic $0$ and $C$ is a projective curve. Let $f: X\dra X$ be a dominant rational self-map of $X$. We choose a line bundle $L\in \mr{Pic}(X)$. Then we can find an integral model $(\mc{X},\mc{L})$ of $(X, L)$, where $\mc{X}$ is projective and flat over $C$, $\mc{X}_K\cong X$ and $\mc{L}_K\cong L$.  For an algebraic point $x\in X(\ov{K})$, the \emph{geometric height} is defined as
\[h_{\mr{geom},\mc{L}}(x):=\frac{1}{[K(x):K]}\mc{L}\cdot \Delta_x,\]
where $\Delta_x$ is the Zariski closure of the image of $x:\Spe \ov{K}\ra X\hookrightarrow \mc{X}$.
Furthermore, if the $f$-orbit of $x$ is well-defined, one can define the arithmetic degree by
\[\alpha_{\mr{geom}}(f,x):=\lim_{n\rightarrow+\infty}h_{\mr{geom},\mc{L}}^+(f^n(x))^{\frac{1}{n}}\]
if the limit exists, where $h_{\mr{geom},\mc{L}}^+(\cdot):=\max\{h_{\mr{geom},\mc{L}}(\cdot),1\}$ for an ample line bundle $\mc{L}$.
However, due to the absence of the Northcott property for this geometric height, the Kawaguchi--Silverman conjecture does not hold. In \cite[Example 3.8]{Matsuzawa2018zero}, two examples are provided. Meanwhile, after replacing the geometric height with Moriwaki height, we obtain a more reasonable definition to play the role of the arithmetic degree, as defined in Definition \ref{def_arithdeg_C} by spreading out.

\begin{exem}\label{exam_func}

    \begin{enumerate}
        \item Consider an endomorphism $f:\mb{P}_{\ov{\mathsf k(t)}}^1\ra \mb{P}_{\ov{\mathsf k(t)}}^1$ with dynamical degree $\lambda_1(f)>1$. Let $P\in \mb{P}^1(\mathsf{k})$ be a non-preperiodic $\mathsf{k}$-point. Then $\alpha_{\mr{geom}}(f,P)=1<\lambda_1(f)$. However, since $f$ is a polarized endomorphism, it follows from Proposition \ref{theo_KSC_C} that $\alpha(f,P)=\lambda_1(f)$.
        \item Consider the rational self-map 
        \[f: \mb{P}_{\ov{\mathsf k(t)}}^2\dra \mb{P}_{\ov{\mathsf k(t)}}^2,\quad [x:y:z]\mapsto [x^2z: y^3: z^3].\]
        Let $P=[t:2:1]\in \mb{P}^2(\mathsf k(t))$, then the $f$-orbit $\mO_f(P)$ is Zariski dense. According to the computation in \cite[Example 3.8]{Matsuzawa2018zero}, we have \[\alpha_{\mr{geom}}(f,P)=2<\lambda_1(f)=3.\] 
        On the other hand, the triple $(\mb{P}_{\ov{\mathsf k(t)}}^2,f, P)$ is defined over $K=\mb{Q}(t)$. Let $(\ov{B};\ov{H})$ be an ample polarization. In particular, we can choose $B=\mb{P}^1_{\mb{Z}}$ and $\ov{H}=(\mathcal O(1),\|\cdot\|_{\mathrm{FS}})$, which is the tautological bundle with the Fubini-Study metric on $\mb{P}^1_{\mb{Z}}$.
        Then the naive height is given by
        \[h_{nv,K}^{\ovl B}([\lambda_0:\lambda_1:\lambda_2])=\sum_{\Gamma \in B^{(1)}}\max_i\{-\mathrm{ord}_\Gamma(\lambda_i)\}\widehat{\deg}(\ov{H}|_{\Gamma})+\int_{B(\C)}\log(\max\{\lvert\lambda_i\rvert\})c_1(\ovl H).\]
    Since $f^n(P)=[t^{2^n}:2^{3^n}:1]$, it follows that
    \[
     h_{nv,K}^{\ovl B}(f^n(P))= \int_{\mb{P}^1(\mb{C})}\log\left(\max\left\{\left|z_0/z_1\right|^{2^n},2^{3^n},1\right\}\right)c_1(\ov{H})\geq 3^n \log 2\int_{\mb{D}}\iota^*c_1(\ov{H})>0,
    \]
    where $\mb{D}$ is the disk $\{|z|\leq 2\}\subset \mb{C}$ and $\iota:\mb{D}\hookrightarrow \mb{P}^1(\mb{C})$ is the inclusion $z\mapsto [z:1]$. Thus, we obtain 
    \[\alpha(f,P)=\lim_{n\ra\infty}\max\left\{h_{nv,K}^{\ovl B}(f^n(P)),1\right\}^{\frac{1}{n}}=3=\lambda_1(f).\]
    \end{enumerate}
\end{exem}

\subsection{Cases over a universal domain}\label{Sec_complex}
In this subsection, we further assume that $\mathbf k$ is an algebraically closed field with $\mr{tr.deg}(\mathbf k/\Q)=\infty$, i.e., $\mathbf k$ is a universal domain. We consider a normal projective variety $X$ and a rational self-map $f: X\dra X$  defined over $\mathbf{k}$. A fundamental question is determining how many points $x\in X(\mathbf{k})$ satisfy the equality $\alpha(f,x)=\lambda_1(f)$.
As the arithmetic degree measures the arithmetic complexity of the $f$-orbit of $x$, it is natural to expect that if $x$ is a purely transcendental point in $X(\mathbf{k})$ (see Definition \ref{def_trans}), then $x$ achieves the maximal arithmetic degree. We will prove that for a very general point $x\in X(\mathbf{k})$, the arithmetic degree $\alpha(f,x)$ is constant. Moreover, when $f$ is an endomorphism, we show that $\alpha(f,x)=\lambda_1(f)$ for very general points $x\in X(\mathbf{k})$. 



In fact, very general points in $X(\mathbf{k})$ are purely transcendental (see Definition \ref{def_trans}). A point is a purely transcendental point if and only if it does not lie in any proper $K$-subvarieties of $X$. Consider the set
\[S:=\{V\subset X\mid V\text{ is a closed subvariety defined over } K\}.\]
Then the set of purely transcendental points with respect to $X_K\ra \Spe K$ is given by $X(\mathbf{k})\setminus \bigcup\limits_{V\in S} V(\mathbf{k})$. 

Now, we can prove the following proposition, which helps us to transfer one purely transcendental point to another.

\begin{prop}\label{prop_trans}
Let $X_K$ be a quasi-projective variety over a subfield $K\subset \mathbf k$ which is finitely generated over $\mb{Q}$. Consider the base change $X=X_K\times_{\Spe K} \Spe {\mathbf k}$.  Then for any purely transcendental points $x,y\in X(\mathbf{k})$ with respect to $X_{K}/K$, there exists a $K$-automorphism $\sigma\in \mr{Gal}(\mathbf{k}/K)$ such that $\sigma(x)=y$.
\end{prop}
\begin{proof}
By the definition of purely transcendental points, $x,y$ are not contained in any subvariety $V(\mathbf{k})$ of $X(\mathbf{k})$, where $V$ is a proper $K$-subvariety of $X$. We divide the proof into three steps. We may assume that $X_K$ is affine, that is, $X_K$ is a closed subset of $\mb{A}^{n}$. We will drop this assumption in \textbf{Step 3} below. In the first two steps, we proceed by induction on the codimension of $X_K$ in $\mb{A}^n$.

\textbf{Step 1:} We first consider the case that $\mr{codim\,}X_K=1$ in $\mathbb A^n$. We assume that $X_K$ is defined by an irreducible polynomial $f(x_1,\dots,x_n)=0$ with coefficients in $K$. We take a Noether's normalization $\pi:X_K\rightarrow \mathbb{A}^{n-1}$ as the composition \begin{align*}
    X_K\hookrightarrow &\kern 2em\mathbb{A}^n\kern2em\longrightarrow \kern2em \mathbb{A}^{n-1}\\ &(x_1,\dots,x_n)\mapsto (x_1',\dots,x_{n-1}')=(x_1-a_1x_n,\dots,x_{n-1}-a_{n-1}x_n),
\end{align*}
where $a_1,\dots,a_{n-1}$ lie in $K$ such that $$f(x_1'+a_1 x_n,\dots,x_{n-1}'+a_{n-1}x_n,x_n)=\lambda x_n^r+(\text{terms of degree}<r\text{ in }x_n).$$ Here $\lambda$ is a polynomial in $a_1,\dots,a_{n-1}$ that is nonzero for an appropriate choice of $a_1,\dots,a_n$. It follows that $K[x_1,\dots,x_n]/(f)$ is a finitely generated module over $K[x_1',\dots,x_{n-1}']$.

Let $x=(\alpha_1,\dots,\alpha_n)$ and $y=(\beta_1,\dots,\beta_n)$. Define $\alpha_i'=\alpha_i-a_i\alpha_n$ and $\beta_i'=\beta_i-a_i\beta_n$ for $i=1,\dots,n-1$.
Given that $x$ does not belong to any closed subvariety of $X$ defined over $K$, it follows that $\pi(x)$ is not contained in any closed subvariety of $\mathbb{A}^{n-1}$ defined over $K$. This implies that $\alpha_1',\dots,\alpha_{n-1}'$ are algebraically independent over $K$. Similarly, $\beta_1',\dots,\beta_{n-1}'$ are algebraically independent over $K$.  By \cite[Page 72]{Lang1972AG}, there exists a $K$-automorphism $\sigma_1\in \mr{Gal}(\mathbf{k}/K)$ such that $\sigma_1(\alpha_i')=\beta_i'$ for $i=1,\dots, n-1$. In other words, $\sigma_1(\pi(x))=\pi(y)$. Then $\pi(\sigma_1(x))=\pi(y)$.  

Let $L=K(\beta_1',\dots,\beta_{n-1}')$.  Now, we write $f(x_1,\dots,x_n)=\widehat{f}(x_1',\dots,x_{n-1}',x_n)$. We consider the polynomial $\widehat{f}(\beta_1',\dots,\beta_{n-1}',x_n)\in L[x_n]$.
Assume it is reducible. Then it can be expressed as $g_1(x_n)g_2(x_n)$, where $g_1,g_2\in L[x_n]$.  We can find $\wt{g}_1,\wt{g}_2\in K[x_1',\dots,x_{n-1}',x_n]$ and $h_1,h_2\in K[x_1',\dots,x_{n-1}']$, such that
\[g_1(x_n)=\frac{\wt{g}_1(\beta_1',\dots,\beta_{n-1}',x_n)}{h_1(\beta_1',\dots,\beta_{n-1}')},
g_2(x_n)=\frac{\wt{g}_2(\beta_1',\dots,\beta_{n-1}',x_n)}{h_2(\beta_1',\dots,\beta_{n-1}')}.\]
Given that $\beta_1',\dots,\beta_{n-1}'$ are algebraically independent over $\ov{K}$, we observe that
\[\widehat{f}h_1h_2=\wt{g}_1\wt{g}_2\in K[x_1',\dots,x_{n-1}',x_n].\]
This is due to the fact that all coefficients of $\widehat{f}h_1h_2-\wt{g}_1\wt{g}_2$ are polynomials in $K[x_1',\dots,x_{n-1}']$ that vanish at $(\beta_1',\dots,\beta_{n-1}')$. The irreducibility of $\widehat{f}$ yields that $\widehat{f}\mid \wt{g}_1$ or $\widehat{f}\mid \wt{g}_2$. This implies $\widehat{f}(\beta_1',\dots,\beta_{n-1}',x_n)\mid g_1(x_n)$ or $g_2(x_n)$, hence a contradiction. Therefore $\widehat{f}(\beta_1',\dots,\beta_{n-1}',x_n)\in L[x_n]$ is irreducible.  Moreover, considering that $\sigma_1(\alpha_n)$ and $\beta_n$ are roots of $\widehat{f}(\beta_1',\dots,\beta_{n-1}',x_n)=0$, there exists a $L$-automorphism $\sigma_2\in \mr{Gal}(\mathbf{k}/L)$ such that $\sigma_2\circ\sigma_1(\alpha_n)=\beta_n$. Consequently, $\sigma_2\circ\sigma_1(x)=y$.

\textbf{Step 2:} If $\mr{codim\,}X_K\geq 2$ in $\mb{A}^n$, say $X_K=V(I)$, where 
\[I=(f_1(x_1,\dots,x_n),\dots,f_m(x_1,\dots,x_n))\] is a prime ideal and all $f_i$ are irreducible polynomial for some $m\in \mb{N}$. Assume that the statement of the proposition is true for any subvariety in some $\mathbb A^\ell(\ell\in \N_+)$ of codimension smaller than $\mr{codim\,}X_K$. We can choose $a_1,\dots,a_{n-1}$ such that the $K$-algebra
$K[x'_1,\dots,x_{n-1}',x_n]/(\widehat{f}_i(x'_1,\dots,x_{n-1}',x_n))$ is a finite $K[x'_1,\dots,x_{n-1}']$-module for each $i=1,\dots,m$, where 
$$\widehat{f}_i(x'_1,\dots,x_{n-1}',x_n)=f_i(x'_1+a_1x_n,\dots,x_{n-1}'+a_{n-1}x_n,x_n).$$
Indeed, we may assume that $f_i=P^{(i)}_{d_i}+(\text{terms of degree<}d_i)$ where $i=1,\dots,m$ and $d_i$ is the degree of $f_i$. Then there exists $(a_1,\dots,a_{n-1},1)\in \mb{A}^n(K)\setminus \left(
\mathop{\bigcup}\limits_{i}V(P_{d_i})\right)(K)$ since $P_{d_i}^{(i)}$ are homogeneous. 

Consider the map \begin{align*}
     &\kern 2em\mathbb{A}^n\kern2em\longrightarrow \kern2em \mathbb{A}^{n-1}\\ &(x_1,\dots,x_n)\mapsto (x_1',\dots,x_{n-1}')=(x_1-a_1x_n,\dots,x_{n-1}-a_{n-1}x_n).
\end{align*}
We denote by $I'=(\widehat{f}_1,\dots,\widehat{f}_m)$. Then the subvariety $\pi(X)\subset \mb{A}^{n-1}$ is defined by $J=I'\cap K[x_1',\dots,x_{n-1}']$.
We can see $\mr{codim\,}\pi(X_K)=\mr{codim\,}X_K-1$. By the induction hypothesis, there exists a $K$-automorphism $\sigma_1\in  \mr{Gal}(\mathbf{k}/K)$ such that $\sigma_1(\pi(x))=\pi(y)$. It follows that $\pi(\sigma_1(x))=\pi(y)$. 

We denote $L=K(\beta_1',\dots,\beta_{n-1}')$. Since $L[x_n]$ is a principal ideal domain, let $g(x_n)$ be the greatest common divisor of $\widehat{f}_i(\beta_1',\dots,\beta_{n-1}',x_n)(i=1,\dots,m)\in L[x_n]$. Suppose that $g$ is reducible, we write $g=g_1g_2,$ where $g_1,g_2\in L[x_n]$. We can find $\wt{g},\wt{g}_1,\wt{g}_2\in K[x_1',\dots,x_{n-1}',x_n]$ and $h,h_1,h_2\in K[x_1',\dots,x_{n-1}']\setminus J$, such that
\begin{align*}
    &g(x_n)=\frac{\wt{g}(\beta_1',\dots,\beta_{n-1}',x_n)}{h(\beta_1',\dots,\beta_{n-1}')},\\
    &g_1(x_n)=\frac{\wt{g}_1(\beta_1',\dots,\beta_{n-1}',x_n)}{h_1(\beta_1',\dots,\beta_{n-1}')},\\
&g_2(x_n)=\frac{\wt{g}_2(\beta_1',\dots,\beta_{n-1}',x_n)}{h_2(\beta_1',\dots,\beta_{n-1}')}.
\end{align*}
If we express $h_1h_2\wt{g}-h\wt{g}_1\wt{g}_2$ in terms of $x_n$, then all its coefficients are polynomials in $K[x_1',\dots,x_{n-1}']$, which vanish at $(\beta_1',\dots,\beta_{n-1}')$. By the assumption of $y$, the minimal closed $K$-subvariety contains $(\beta_1',\dots,\beta_{n-1}')$ is $\pi(X)$, generated by $J$, we have
\[ h_1h_2\wt{g}-h\wt{g}_1\wt{g}_2\in J[x_n]\subset I'.\]
Since $g(x_n)$ can be expressed by a linear combination of $\widehat{f}_i(\beta_1',\dots,\beta_{n-1}',x_n)(i=1,\dots,m)$, we have $\wt{g}\in I'$ by the same argument. Thus $\wt{g}_1\wt{g}_2\in I'$. Since $I'$ is prime, let us assume $\wt{g}_1\in I'$. Then one can write $\wt{g}_1=u_1\widehat{f}_1+\dots+u_m\widehat{f}_m$. It follows that
\[g_1(x_n)=\frac{\wt{g}_1(\beta_1',\dots,\beta_{n-1}',x_n)}{h_1(\beta_1',\dots,\beta_{n-1}')}=\sum_{i=1}^m\frac{u_i(\beta_1',\dots,\beta_{n-1}',x_n)}{h_1(\beta_1',\dots,\beta_{n-1}')}\widehat{f}_i(\beta_1',\dots,\beta_{n-1}',x_n).\]
Therefore $g(x_n)\mid g_1(x_n)$, which leads to a contradiction. Thus $g\in L[x_n]$ is irreducible. 

 Since $\sigma_1(\alpha_n)$ and $\beta_n$ are zeros of $g(x_n)=0$, we can find a $L$-automorphism $\sigma_2\in \mr{Gal}(\mathbf{k}/L)$ such that $\sigma_2\circ\sigma_1(\alpha_n)=\beta_n$. In other words, $\sigma_2\circ\sigma_1(x)=y$.

\textbf{Step 3:} For a general quasi-projective variety $X_K$ and purely transcendental points $x,y$, we choose open affine subsets $U_1,U_2\subset X_K$ such that:
\begin{enumerate}
    \item[(a)] $x_0:=x\in U_1(\mathbf{k}),x_2:=y\in U_2(\mathbf{k}).$    \item[(b)] There exists $x_1\in (U_1\cap U_2)(\mathbf{k})$ which is also purely transcendental.
\end{enumerate}
We can find $\sigma_i\in \mr{Gal}(\mathbf{k}/K)$ such that $\sigma_i(x_{i-1})=x_i$ for $i=1,2$. Set $\sigma=\sigma_2\circ\sigma_1$. Then $\sigma(x)=y$.
\end{proof}

\begin{theo}\label{theo_arithdeg_Galois_orbit}
Let $X$ be a normal quasi-projective variety and $f: X\dra X$ be a dominant rational self-map over $\mathbf{k}$. Suppose $(X,f)$ are defined over a finitely generated field $K$. For any $x,y\in X_f(\mathbf{k})$, if there exists $\sigma\in \mr{Gal}(\mathbf{k}/K)$ such that $\sigma(x)=y$, then we have \begin{equation*}\label{eq_arith_deg_cons}
    \ov{\alpha}(f,x)=\ov{\alpha}(f,y)\text{ and }\underline{\alpha}(f,x)=\underline{\alpha}(f,y).
\end{equation*}
\end{theo}
\begin{proof}
Suppose the triple $(X,f,x)$ is defined over $K_1\subset \mathbf{k}$, which is a finitely generated extension of $K$. Then $(X,f,y)$ is defined over $K_2=\sigma(K_1)$. Let $\ov{B}=(B,\ov{H})$ be an ample polarization of $K_1$, as described in \S\ref{Sec_height}. Since $\sigma:K_1\ra K_2$ is a $K$-isomorphism, $\ov{B}$ is also an ample polarization of $K_2$. We take an embedding $X\hookrightarrow \mb{P}^m_{K_1}$. Suppose $x=[\lambda_0:\dots:\lambda_m]\in \mb{P}^m(K_1)$, then consider the naive height
\[\begin{aligned}
    h_{nv,K_1}^{\ovl B}([\lambda_0:\dots:\lambda_m])&=\sum_{\Gamma \in B^{(1)}}\max_i\{-\mathrm{ord}_\Gamma(\lambda_i)\}\widehat c_1(\ovl H_1|_\Gamma)\cdots \widehat c_1(\ovl H_e|_\Gamma)\\
    &\kern 5em+\int_{B(\C)}\log(\max\{\lvert\lambda_i\rvert\})c_1(\ovl H_1)\cdots c_1(\ovl H_e)
\end{aligned}.\]
Since $\mathrm{ord}_\Gamma(\lambda_i)=\mathrm{ord}_\Gamma(\sigma(\lambda_i))$ for all $\Gamma\in B^{(1)}$ and $|\lambda_i(z)|=|\sigma(\lambda_i)(z)|$ for all $z\in B(\mathbf{k})$, it follows that
\[h_{nv,K_1}^{\ovl B}([\lambda_0:\dots:\lambda_m])=h_{nv,K_2}^{\ovl B}([\sigma(\lambda_0):\dots:\sigma(\lambda_m)]).\]
Similarly, we have $h_{nv,K_1}^{\ovl B}(f^n(x))=h_{nv,K_2}^{\ovl B}(f^n(y)), n\in \mb{Z}_+$. Consequently, we obtain that $\ov{\alpha}(f,x)=\ov{\alpha}(f,y)$ and $\underline{\alpha}(f,x)=\underline{\alpha}(f,y)$.
\end{proof}

\begin{coro}\label{coro_arithdeg_constant}
Let $X$ be a normal projective variety and $f: X\dra X$ be a dominant rational self-map over $\mathbf{k}$. Suppose $(X,f)$ is defined over a finitely generated field $K$ over $\mb{Q}$. For $x,y\in X_f(\mathbf{k})$ satisfying $Z(x)_{K}=Z(y)_{K}$, we have 
\[\ov{\alpha}(f,x)=\ov{\alpha}(f,y) \text{ and } \underline{\alpha}(f,x)=\underline{\alpha}(f,y),\]
where $Z(x)_K$ and $Z(y)_K$ are defined in Definition \ref{def_trans}.
In particular, there exists a countable set of proper subvarieties of $X$:
\[S=\{V\subset X\mid V\text{ is a closed subvariety defined over }K\}\cup \{I(f^n)\mid n\in\mb{Z}_+\},\]
 such that for all $x\in X(\mathbf{k})\setminus \bigcup\limits_{V\in S}V(\mathbf{k})$, the arithmetic degrees $\ov{\alpha}(f,x)$ and $\underline{\alpha}(f,x)$ are constants.
\end{coro}
\begin{proof}
Suppose $X$ and $f$ are defined over a finitely generated field $K$ over $\mb{Q}$. Then there exists a normal projective variety $X_K$ over $K$ such that $X=X_K\times_{\Spe K}\Spe \mathbf{k}$. Denote $Y=Z(x)$ and $Y_{K}=Z(x)_{K}$, then $x,y$ are purely transcendental points with respect to $Y_{K}/K$. Applying Proposition \ref{prop_trans}
to $Y$, there exists a $K$-automorphism $\sigma\in \mr{Gal}(\mathbf{k}/K)$ such that $\sigma(x)=y$. Then by Theorem \ref{theo_arithdeg_Galois_orbit}, their arithmetic degrees are equal. 

Let $S_1=\{V\subset X\mid V\text{ is a closed subvariety defined over }K\}.$
Let $I(f)$ be the indeterminacy locus of $f$. We define $S_2=\{I(f^n)\mid n\in\mb{Z}_+\}.$ Now let $S=S_1\cup S_2$, which forms a countable set. $X(\mathbf{k})\setminus \bigcup\limits_{V\in S}V(\mathbf{k})$ consists of purely transcendental points with respect to $X_{K}\ra \Spe K$. According to the argument above, their arithmetic degrees are equal. 
\end{proof}

\begin{rema}
A direct consequence of Corollary \ref{coro_arithdeg_constant} is that the set
\[\left\{\ov{\alpha}(f,x)\mid x\in X_f(\mathbf{k})\right\}\cup \left\{\underline{\alpha}(f,x)\mid x\in X_f(\mathbf{k})\right\}\]
is countable. This is because there are only countably many subvarieties $Z(x)_K$ for a fixed finitely generated field $K$ over $\mb{Q}$. If $f$ is an endomorphism, it is a finite set of algebraic integers, see Proposition \ref{prop_finite_set}.
\end{rema}

For a rational self-map $f: X\dra X$, determining whether $\alpha(f,x)=\lambda_1(f)$ for $x\in X(\mathbf{k})$ is generally difficult. However, in the case where $f$ is an endomorphism, the limit defining $\alpha(f,x)$ exists (see Definition \ref{arith_deg_fgf}). Furthermore, we can find a point $x_0\in X(\mathbf{k})$ which attains the maximal arithmetic degree. By combining this result with Corollary \ref{coro_arithdeg_constant}, we can show that $\alpha(f,x)=\lambda_1(f)$ for very general $x\in X(\mathbf{k})$. In other words,  the equality $\alpha(f,x)=\lambda_1(f)$ holds for $x\in X(\mathbf{k})$ outside a countable union of proper subvarieties.

\begin{lemm}\label{lemm_boundeddeg}
Let $X$ be a quasi-projective variety over $\mathbf{k}$. Suppose $X$ is defined over a subfield $K\subset \mathbf{k}$. Then for any $L=K(\beta)\subset\mathbf{k}$ satisfying $\beta\notin \ov{K}$, there exists an integer $N>0$ such that
\[\{x\in X_L(\ov{L})\mid x \text{ is transcendental with respect to }X_K\ra \Spe K, [L(x):L]<N\}\]
has infinitely many elements.
\end{lemm}
\begin{proof}
Without loss of generality, we assume $X=\Spe R$ is affine, where $R=\mathbf{k}[x_1,\dots,x_n]/(f_1,\dots,f_m)$ and $f_1,\dots,f_m\in K[x_1,\dots,x_n]$. By definition, we have $X_K=\Spe K[x_1,\dots,x_n]/(f_1,\dots,f_m)$ and $X_L=\Spe L[x_1,\dots,x_n]/(f_1,\dots,f_m)$.
Suppose $\dim X=d$, then by Noether's normalization, there exists a surjective finite morphism $f_K: X_K\ra \mb{A}^d_K$, which induces a surjective finite morphism $f: X\ra \mb{A}^d$. There exists an integer $N>0$ such that the set
\[\{x\in \mb{A}^d(L)\mid x\text{ is transcendental with respect to }\mb{A}^d_K\ra \Spe K\}\]
has infinite elements. Since the morphism $f_K$ is finite, $R_L=L[x_1,\dots,x_n]/(f_1,\dots,f_m)$ is a finitely generated $L[y_1,\dots,y_d]$-module, which is generated by $r$ elements for some integer $r$. For any $x\in X(\ov{L})$ with image $y=f(x)$, we have
\[ [L(x):L]\leq r[L(y):L].\]
Furthermore, if $y$ is transcendental with respect to $\mb{A}^d_K\ra \Spe K$, it follows that $x$ is transcendental with respect to $X_K\ra \Spe K$.
Thus the set
\[\{x\in X_L(\ov{L})\mid x \text{ is transcendental with respect to }X_K\ra \Spe K, [L(x):L]<r\}\]
has infinitely many elements.
\end{proof}

\begin{theo}\label{theo_verygeneral}
Let $X$ be a normal projective variety and $f: X\ra X$ be an endomorphism over $\mathbf{k}$. Suppose $X$ and $f$ are defined over a finitely generated field $K\subset \mathbf{k}$ over $\mb{Q}$. Then 
for any purely transcendental point $x\in X(\mathbf{k})$ with respect to $X_{\ov{K}}\ra \Spe \ov{K}$, we have $\alpha(f,x)=\lambda_1(f)$. In particular, there exist countably many proper closed subvarieties $\{V_i\mid i\in\mb{Z}_+\}$ of $X$ such that $\alpha(f,x)=\lambda_1(f)$ for all $x\in X(\mathbf{k})\setminus \bigcup\limits_{i\in \mb{Z}_+}V_i(\mathbf{k})$.
\end{theo}
\begin{proof}
We may assume $\lambda_1(f)>1$.

Suppose $X$ and $f$ are defined over a finitely generated field $K$ over $\mb{Q}$. Let 
\[S:=\{V\subset X\mid V\text{ is a closed subvariety defined over }K\}.\]
For all $x,y\in X(\mathbf{k})\setminus \bigcup\limits_{V\in S}V(\mathbf{k})$, we have $\alpha(f,x)=\alpha(f,y)$. It suffices to find $x_0\in X(\mathbf{k})\setminus \bigcup\limits_{V\in S}V(\mathbf{k})$ such that $\alpha(f,x_0)=\lambda_1(f)$.

Firstly, we choose ample divisors $H_1,\dots,H_{d-1}$ over $X$ such that $C=H_1\cap \dots \cap H_{d-1}$ is an irreducible curve and $C$ is not contained in any subvariety $V\in S$.  Since $f^*$ preserve the nef cone, there exists a nef $\mb{R}$-divisor $D$ such that $f^*D\equiv \lambda_1(f)D$. Then we have $(D\cdot H_1\cdots H_{d-1})>0$ by \cite[Lemma 20]{KS2016dyna}. Thus $D|_C$ is ample.

Suppose $X,f, C$ and $D$ are defined over a finitely generated field $K_1$ over $\mb{Q}$.
The set $C(\mathbf{k})\cap \bigcup\limits_{V\in S}V(\mathbf{k})$ is countable. Suppose
\[C(\mathbf{k})\cap \bigcup\limits_{V\in S}V(\mathbf{k})=\{x_1,x_2,\dots\},\]
where $x_1,x_2,\dots$ are defined over $K_2$, which is a countably generated extension over $K_1$. 
Since the cardinality of $\ov{K_2}$ is countable, we can choose an element $\beta\in \mathbf{k}\backslash\ov{K_2}$. Then we have $\ov{K_1(\beta)}\cap \ov{K_2}=\ov{K_1}$. Let $K_3=K_1(\beta).$

We fix an ample polarization $\ov{B}=(B,\ov{H})$ of $K_3$ and a Moriwaki height
$h_{\ov{D}}$. For simplicity, we continue to denote $D_{K_3}$ by $D$ here. Since $D|_C$ is ample, we may assume $h_{\ov{D}}|_{C_{K_3}}\geq 1$. Consider the canonical height
\[\widehat{h}_D:=\lim_{n\ra\infty}\frac{h_{\ov{D}}\circ f_{K_3}^n}{\lambda_1(f_{K_3})^n}.\]
By Corollary \ref{coro_canonicalheight}, the limit converges and we have
\[\widehat{h}_D=h_{\ov{D}}+O(\sqrt{h_{\ov{A}}})\]
for an ample divisor $A$ and a height function $h_{\ov A}\geq 1$. Thus there exist constants $0<M_1<M_2$ such that 
\[\widehat{h}_D(x)\geq  h_{\ov{D}}(x)-M_1\sqrt{h_{\ov{A}}(x)}\geq h_{\ov{D}}(x)-M_2\sqrt{h_{\ov{D}}(x)}, \forall x\in C_{K_3}(\ov{K_3}).\]
If $\widehat{h}_D(x)>0$, then $\alpha(f,x)=\lambda_1(f)$. It follows that
\[
\begin{aligned}
    \{x\in C_{K_3}(\ov{K_3})\mid \alpha(f,x)<\lambda_1(f)\}&\subset \{x\in C_{K_3}(\ov{K_3})\mid \widehat{h}_D(x)\leq 0\}\\
    &\subset  \{x\in C_{K_3}(\ov{K_3})\mid h_{\ov{D}}(x)-M_2\sqrt{h_{\ov{D}}(x)}\leq 0\}\\
    &\subset  \{x\in C_{K_3}(\ov{K_3})\mid h_{\ov{D}}(x)\leq M_2^2\}.
\end{aligned}\]
According to Lemma \ref{lemm_boundeddeg}, there exists a constant $N>0$ such that the set
\[T=\{x\in C(\ov{K_3})\mid  x \text{ is transcendental with respect to }C_{K_1}\ra \Spe K_1, [K_3(x):K_3]<N\}\]
contains infinitely many elements. By the Northcott property, there exists $x_0\in T$ such that $h_{\ov{D}}(x_0)>M^2$, which implies $\alpha(f,x_0)=\lambda_1(f)$. Moreover, the point $x_{0,\mathbf{k}}\in X(\mathbf{k})$ that is induced by $x_0$ is not contained in $C(\mathbf{k})\cap \bigcup\limits_{V\in S}V(\mathbf{k})$, as this would imply that $x_0$ is defined over $\ov{K_2}\cap \ov{K_3}=\ov{K_1}$, contradicting the transcendence of $x_0$. Consequently, $\alpha(f,x)=\lambda_1(f)$ for all $x\in X(\mathbf{k})\setminus \bigcup\limits_{V\in S}V(\mathbf{k})$.
\end{proof}

\begin{rema}
If $\mathbf{k}$ is an uncountable field of characteristic $0$, the Zariski dense orbit conjecture holds as proved in \cite{AC2008Fibre}. Specifically, assuming $\mathbf{k}(X)^f=\mathbf{k}$, then there exist countably many proper closed subvarieties $\{V_i\mid i\in\mb{Z}_+\}$ of $X$ such that for all $x\in X(\mathbf{k})\setminus \bigcup_{i}V_i(\mathbf{k})$, the $f$-orbit $\mO_f(x)$ is Zariski dense. Moreover, if we further assume the Kawaguchi--Silverman conjecture holds, we also conclude that $\alpha(f,x)=\lambda_1(f)$ for all $x\in X(\mathbf{k})\setminus \bigcup_{i}V_i(\mathbf{k})$.
\end{rema}

\end{document}